\documentclass{article}
\usepackage{amsmath}
\usepackage{amsfonts}
\usepackage{amsthm}
\usepackage{amssymb, setspace}
\usepackage{mathtools}
\usepackage{lipsum}
\usepackage{color,graphicx,epsfig,geometry,fancyhdr,hyperref}
\usepackage[T1]{fontenc}
\usepackage{float}
\usepackage{subfig}
\usepackage{commath}
\usepackage{bbm}
\usepackage{mathrsfs,fleqn}
\usepackage{pgfplots}
\pgfplotsset{compat=newest}
\newtheorem{theorem}{Theorem}[section]

\newtheorem{lemma}{Lemma}[section]

\newcommand{\N}{\mathbb{N}}
\newcommand{\s}{\mathbb{S}}
\newcommand{\Z}{\mathbb{Z}}
\newcommand{\weakc}{\rightharpoonup}
\newcommand{\R}{\mathbb{R}}
\newcommand{\C}{\mathbb{C}}

\newcommand{\dnu}{\partial_\nu}

\newcommand{\dnuy}{\partial_{\nu_y}}
\newcommand{\pdr}{\partial_r}
\newcommand{\dd}{\mathrm{d}}

\newcommand{\ov}{\overline}

\newcommand{\paren}[1]{\left( #1 \right) }
\newcommand{\exD}{\mathbb{R}^2 \setminus \overline{D}}

\usepackage{comment}
\usepackage{ulem}
\graphicspath{{paper-pics/}}

\begin{document}

\begin{flushleft}
\Large 
\noindent{\bf \Large On the transmission eigenvalues for scattering by a clamped planar region}
\end{flushleft}

\vspace{0.2in}

{\bf  \large Isaac Harris and Heejin Lee}\\
\indent {\small Department of Mathematics, Purdue University, West Lafayette, IN 47907 }\\
\indent {\small Email: \texttt{harri814@purdue.edu} and  \texttt{lee4485@purdue.edu} }\\

{\bf  \large Andreas Kleefeld}\\
\indent {\small Forschungszentrum J\"{u}lich GmbH, J\"{u}lich Supercomputing Centre, } \\
\indent {\small Wilhelm-Johnen-Stra{\ss}e, 52425 J\"{u}lich, Germany}\\
\indent {\small University of Applied Sciences Aachen, Faculty of Medical Engineering and } \\
\indent {\small Technomathematics, Heinrich-Mu\ss{}mann-Str. 1, 52428 J\"{u}lich, Germany}\\
\indent {\small Email: \texttt{a.kleefeld@fz-juelich.de}}\\

\begin{abstract}
\noindent In this paper, we consider a new transmission eigenvalue problem derived from the scattering by a clamped cavity in a thin elastic material. Scattering in a thin elastic material can be modeled by the Kirchhoff--Love infinite plate problem. This results in a biharmonic scattering problem that can be handled by operator splitting. The main novelty of this transmission eigenvalue problem is that it is posed in all of $\mathbb{R}^2$. This adds analytical and computational difficulties in studying this eigenvalue problem. Here, we prove that the eigenvalues can be recovered from the far field data as well as discreteness of the transmission eigenvalues.  We provide some numerical experiments via boundary integral equations to demonstrate the theoretical results. We also conjecture monotonicity with respect to the measure of the scatterer from our numerical experiments. 
\end{abstract}



\section{Introduction}

Here, we study a transmission eigenvalue problem associated with scattering by a clamped region in a thin elastic material. The inverse scattering problem to determine an unknown scatterer from time--harmonic acoustic and electromagnetic waves has received a lot of attention. This is due to the wide applications in non--destructive testing. However, the inverse scattering problem in a thin elastic plate, which is modeled by the biharmonic wave equation has only recently received some attention \cite{near-lsmBH,iterative-BH23,DongLi24,DongLi-Unique,LSM-BHclamped,DSM-BH24}. These problems have significant applications in the area of detecting geomagnetic defects \cite{applicationref2,applicationref3} as well as medical imaging \cite{applicationref4}. We will consider the direct Kirchhoff--Love infinite plate problem for an impenetrable clamped scatterer to derive the associated `clamped' transmission eigenvalue problem. This is analogous to what has been done in acoustic scattering. We provide analytically and numerical study of the associated eigenvalue problem.

As we will see in the next section, this eigenvalue problem is posed as a system of partial differential equation in all of $\R^2$. Note that from an application stand point, we consider the problem in $\R^2$, but all of our theoretical results are also valid in $\R^3$. The fact that this transmission eigenvalue problem is posed in an infinite domain makes its investigation more difficult than standard transmission eigenvalue problems. Due to the fact that the problem is posed in an infinite domain, we propose using boundary integral equations for our numerical examples. This is useful for the fact that the problem is reduced to the boundary of the scatterer. By applying the boundary conditions and using the well known jump relations for the double and single layer potential gives a 2$\times$2 system of boundary integral equations. After discretization, this gives a non-linear matrix eigenvalue problem. By appealing to a technique developed in \cite{beyn} and used in \cite{ cakonikress, Kleefeld, ComplexTrajectories}, ne can approximate the clamped transmission eigenvalues for a region and additionally compute numerically the corresponding eigenfunctions inside and outside of the given region.

In our investigation, we consider the discreteness of the transmission eigenvalues as well as prove that they can be recovered from the so--called far field data. The discreteness result has applications to the inverse shape problem as the linear sampling method \cite{LSM-BHclamped} for recovering the scatterer is not valid when the wave number is a transmission eigenvalue. From our analysis, we see that in a subset of the complex plane there can only be real--valued transmission eigenvalues. We also prove that the eigenvalues can be recovered with the measured far field data. This implies that these eigenvalues can be recovered as long as one knows the approximate location of the scatterer. This result is useful when considering the inverse spectral problem (see for e.g. \cite{ap-invTE,bck-invTE,homogeniz-TE,gp-invTE})  of inferring information about the scatterer from the transmission eigenvalues. Here, we provide some numerical evidence that these eigenvalues are monotone decreasing with respect to the measure of the scatterer. This would be an interesting analog to the acoustic sound soft where the associated eigenvalue problem is the Dirichlet eigenvalue problem for the scatterer. It is well known that the Dirichlet eigenvalue are monotone decreasing with respect to the measure of the domain.

The rest of the paper is organized as follows. In the next section, we discuss the direct scattering problem and derive the associated `clamped' transmission eigenvalue problem. With this, in the following section, we then turn our attention to proving the discreteness of the transmission eigenvalues. To this end, we consider the equivalent variational formulation and appeal to the analytic Fredholm theorem to prove the claim. In the preceding section, we prove that the eigenvalues can be recovered from the far field data as was first done in \cite{cchlsm} as well as provide an analytic example for the unit disk. Lastly, we provide some numerical examples of computing the transmission eigenvalues and corresponding eigenfunctions for different scatterers. We also demonstrate how to recover the eigenvalues from the far field data and compare this to the computed eigenvalues. Finally, a conclusion and outlook is given in the last section.

\section{Formulation of Transmission Eigenvalue Problem}
In order to derive the transmission eigenvalue problem under consideration, we first consider the corresponding direct scattering problem. Therefore, we discuss the Kirchhoff--Love infinite plate problem with a clamped cavity. For this paper, the cavity is modeled by a bounded region $D \subset \R^2$ with an analytic boundary $\partial D$. The cavity is assumed to be illuminated by a time--harmonic incident plane wave, which we denote as $u^{\text{inc}}(x) = \text{e}^{\text{i}kx\cdot d}$, where $d$ is the corresponding incident direction on the unit circle $\mathbb{S}^1=\{x\in \R^2: |x|=1\}$. The Kirchhoff--Love models the `scattering' in a thin elastic material and the total displacement $u$ consists of the scattered field $u^{\text{scat}}$ and the incident field $u^{\text{inc}}$, i.e., $u \coloneqq u^{\text{scat}} + u^{\text{inc}}$. Moreover, we have that the biharmonic scattering problem in $\mathbb{R}^2 \setminus \overline{D}$ for a fixed wave number $k>0$, is modeled by 
\begin{align}\label{biharmonic}
\Delta^2 u^{\text{scat}} - k^4 u^{\text{scat}} = 0 \quad  \text{in } \mathbb{R}^2\setminus \overline{D}, 
\end{align}
with the boundary conditions 
\begin{align}\label{cbc}
  u^{\text{scat}} \big|_{\partial D}=-u^{\text{inc}}  \quad \text{ and } \quad \dnu{u^{\text{scat}}} \big|_{\partial D}=- \dnu{u^{\text{inc}}}, 
\end{align}
where $\nu$ denotes the outward unit normal vector on the boundary $\partial D$. 
To complete the system, we require that the scattered field $u^{\text{scat}}$ and $\Delta u^{\text{scat}}$ satisfies the Sommerfeld radiation condition at infinity (cf. \cite{poh-invsource}): 
\begin{align}\label{SRC}
\lim_{r \to \infty} \sqrt{r}(\pdr{u^{\text{scat}}} - \text{i} k u^{\text{scat}}) = 0 \quad \text{ and } \quad \lim_{r \to \infty} \sqrt{r}(\pdr{ \Delta u^{\text{scat}}} - \text{i} k \Delta u^{\text{scat}})=0, \quad r=|x|,
\end{align}
which is assumed to hold uniformly in $\hat{x}=x/|x|$. Note that the well--posedness of \eqref{biharmonic}--\eqref{SRC} has been established in \cite{DongLi24}(see pp. 9--12).

As in \cite{DongLi24}(see pp. 4), we consider two auxiliary functions $u_H$ and $u_M$ such that
\begin{align}\label{vhvm}
u_H = -\frac{1}{2k^2}\left(\Delta u^{\text{scat}} - k^2 u^{\text{scat}} \right) \quad  \text{ and } \quad u_M = \frac{1}{2k^2}\left(\Delta u^{\text{scat}} + k^2 u^{\text{scat}}\right),
\end{align}
which implying that $u^{\text{scat}} = u_H + u_M$. Then, $u_H$ satisfies the Helmholtz equation in $\exD$ and $u_M$ satisfies the modified Helmholtz equation (i.e. with wave number=$\text{i}k$) in $\exD$. Therefore, the biharmonic wave scattering problem \eqref{biharmonic}--\eqref{SRC} is equivalent to the following problem:
\begin{align}
\Delta u_H + k^2 u_H = 0 \quad \text{in } \exD  \quad &\text{ and } \quad  \Delta u_M - k^2 u_M = 0 \quad \text{in } \exD \label{vhvmeq1} \\
u_H+u_M = -u^{\text{inc}}   \quad &\text{ and } \quad \dnu(u_H+u_M) = -\dnu u^{\text{inc}} \quad \text{on } \partial D  \label{vhvmeq2}.
\end{align}
Along with, the Sommerfeld radiation conditions
\begin{align}\label{SRC1}
\lim_{r \to \infty} \sqrt{r}(\pdr{u_H} - \text{i} k u_H) = 0 \quad \text{ and } \quad \lim_{r \to \infty} \sqrt{r}(\pdr{ u_M} - \text{i} k u_M)=0.
\end{align}
Recall that the radiating fundamental solution $\Phi_k(x,y) $ of the Helmholtz equation in $\R^2\setminus \{x\neq y\}$ is given by 
\begin{align}\label{fundsol}
\Phi_k(x,y) = \frac{\text{i}}{4} H^{(1)}_0(k|x-y|) \quad \text{ for all } \quad x \neq y,
\end{align}
where $H^{(1)}_0$ is the Hankel function of the first kind of order $0$. Note that $\Phi_{{\mathrm{i}k}}(x,y)$ is the radiating fundamental solution to the modified Helmholtz equation.
By Green's representation theorem, we have
$$u_H(x)= \int_{\partial D} u_H(y) \dnuy \Phi_k(x,y) - \dnuy u_H(y) \Phi_k(x,y) \dd s(y), \: x \in \exD, $$
and 
$$u_M(x)= \int_{\partial D} u_M(y) \dnuy \Phi_{{\mathrm{i}k}}(x,y) - \dnuy u_M(y) \Phi_{\text{i}k}(x,y) \dd s(y), \: x \in \exD.$$
Since the scattered field $u^{\text{scat}}$ is radiating, it has the asymptotic behavior
\begin{align}\label{asymp}
u^{\text{scat}}(x, d) = \frac{\text{e}^{\text{i} \pi/4}}{\sqrt{8\pi k}}\cdot \frac{\text{e}^{{\mathrm{i}k}|x|}}{\sqrt{|x|}}u^\infty(\hat{x},d) + \mathcal{O}(|x|^{-3/2}), \quad \text{as } |x| \to \infty,
\end{align}
where $u^\infty(\hat{x}, d)$ is the far field pattern of $u^{\text{scat}}$. As mentioned in \cite{DongLi24}, one can derive
\begin{align}\label{ffp1}
u^\infty(\hat{x},d) &= \int_{\partial D} u_H(y,d) \dnuy \text{e}^{-\text{i}k\hat{x}\cdot y} -\dnuy u_H(y,d) \text{e}^{-\text{i}k\hat{x}\cdot y}\dd s(y),
\end{align}
where $\text{e}^{-\text{i}k\hat{x}\cdot y}$ is the far field pattern of the fundamental solution $\Phi_k (x,y)$. In the above expression, we make the dependance on the incident direction $d$ explicit. One interesting fact is that even though the far field data does not contain information from $u_M$ directly, it is still able to uniquely recover the scatterer \cite{DongLi-Unique}.

Notice that this implies that $u^\infty = u_H^\infty$, i.e. the far field pattern for the Helmholtz part of the solution. This is due to the fact that for $k>0$ (also for complex wave number such that Re$(k)>0$) we have that $u_M$ and $\pdr u_M$ decay exponentially as $r \to \infty$ see for e.g. \cite{DongLi24,DSM-BH24}. We can now define the so--called far field operator $F: L^2(\mathbb{S}^1) \to L^2(\mathbb{S}^1)$ given by
\begin{align}\label{ffop}
(Fg)(\hat{x}) = \int_{\mathbb{S}^1} u^\infty (\hat{x}, d) g(d) \, \dd s(d).
\end{align}
With this, we will now derive a new transmission eigenvalue problem associated with the scattering problem \eqref{vhvmeq1}--\eqref{SRC1}. To this end, we take the Herglotz wave function $v_g$ as the incident field in the direct problem, where 
\begin{align}\label{herglotz}
v_g(x) \coloneqq \int_{\s^1} \text{e}^{\text{i}kx \cdot d} g(d)\, \dd s(d).
\end{align}
The Herglotz wave function can be seen as a superposition of incident plane waves. Just as the original incident field $u^{\text{inc}}$ the Herglotz wave function satisfies the Helmholtz equation in $\R^2$. 

Now, let the total field associated with the incident wave $v_g$ be denoted by $u_g \coloneqq u_g^{\text{scat}} + v_g$, where $u_g^{\text{scat}} = u_{H,g} + u_{M,g}$ is the corresponding scattered field. It is well--known that by superposition the far field operator satisfies
$$Fg=u^\infty_g \,\, \text{where $u^\infty_g$ is the corresponding far field pattern of $u^{\text{scat}}_g$}.$$ 
We assume that $Fg=0$, this would imply that $u^\infty_g = u^\infty_{H,g} = 0$ and therefore $u_{H,g} = 0$ in $\exD$ by Rellich's Lemma \cite{Cakoni-Colton-book}. By appealing to \eqref{vhvmeq1}--\eqref{vhvmeq2} we have that by letting $v \coloneqq v_g$ and $w \coloneqq u_{M,g}$, we can derive the transmission eigenvalue problem: $k$ is a transmission eigenvalue if there exists a nontrivial solution $(v, w)$ such that
\begin{align}
\Delta v + k^2 v = 0 \quad \text{in } D \quad &\text{and}  \quad   \Delta w - k^2 w = 0 \quad \text{in } \exD,  \label{tep1} \\
v+w = 0 , \quad \dnu(v+w) = 0 \quad \text{on } \partial D \quad &\text{and } \quad   
\lim_{r \to \infty} \sqrt{r}(\pdr{w} - \text{i}kw) = 0. \label{tep2}
\end{align}
Again, if we assume that the wave number satisfies that Re$(k)>0$, then this implies that $w$ and $\pdr{w} $ decay exponentially, therefore the radiation condition is redundant. With this, we assume that our eigenfunctions $(v, w) \in X(D)$ where 
\begin{align}\label{spacex}
X(D)\coloneqq \left\{ (v, w)\in H^1(D) \times H^1(\exD): (v+w)|_{\partial D}=0 \right\}.
\end{align}
Here, $X(D)$ is the Hilbert space with the standard product space norm. 

In the preceding sections, we will study the clamped transmission eigenvalue problem \eqref{tep1}--\eqref{tep2} analytically and numerically. This problem is analogous to what has been studied for acoustic scattering (see for e.g. \cite{cchlsm,mypaper1,sun-reconTE}).  For the analogous acoustic scattering problem (sound soft), the corresponding associated eigenvalue problem is the Dirichlet eigenvalues of the negative Laplacian for the region $D$. For the simple Dirichlet eigenvalue problem, we have a wealth of knowledge whereas the new eigenvalue problem in \eqref{tep1}--\eqref{tep2} has yet to be studied. We also note that this eigenvalue problem was also recently derived in \cite{LSM-BHclamped} when studying the linear sampling method for recovering a clamped obstacle from the far field operator.

\section{Discreteness of Transmission Eigenvalues}\label{TEdiscrete}
In this section, we study the discreteness of the transmission eigenvalue problem \eqref{tep1}--\eqref{tep2} associated with the scattering problem \eqref{biharmonic}--\eqref{cbc} (or equivalently \eqref{vhvmeq1}--\eqref{vhvmeq2}). In order to study this new eigenvalue problem we are influenced by the analysis in \cite{Cakoni-TE}. In general, transmission eigenvalue problems are non--linear and non self--adjoint (see for e.g. \cite{te-2cbc,ColtYj,ColtYjShixu,ComplexTrajectories}), which makes studying them mathematically and computationally challenging. The analysis in \cite{Cakoni-TE} (see also \cite{cgh-TE,aniso-periodic}) gives a systematic way to prove the existence of the real transmission eigenvalues. In many cases, complex eigenvalues can be computed numerically \cite{te-2cbc}, but their existence is harder to prove.

For analytical considerations in this section we will assume that the wave number 
$$k \in \C_{+}  \quad \text{where} \quad \C_{+} := \{ z \in \C \,\, : \,\,  \text{Re}(z)>0 \quad \text{for all} \quad \text{Re}(z^2)>0\}.$$
Notice that this assumption implies that the eigenfunction $w \in H^1(\exD)$. For our analysis, we will consider the equivalent variational formulation of \eqref{tep1}--\eqref{tep2}. To this end, notice that the eigenvalue problem can be written as the equivalent variational formulation: find $(v,w) \in X(D)$ such that
\begin{align}\label{TEeq}
\mathcal{A}_k \big((v,w) \, ; (\psi_1, \psi_2) \big) - k^2 \mathcal{B}\big((v,w) \, ; (\psi_1, \psi_2) \big) = 0 \quad \text{for all} \quad (\psi_1, \psi_2 ) \in X(D), 
\end{align}
where the sesquilinear form $\mathcal{A}_k: X(D)\times X(D) \to \mathbb{C}$ is given by 
\begin{align}\label{seqformAk}
\mathcal{A}_k\big((v,w) \, ; (\psi_1, \psi_2) \big) = \int_{D}\nabla{v}\cdot \nabla{\overline{\psi_1}} \, \dd x +\int_{\exD} {\nabla{w}\cdot \nabla{\ov{\psi_2}} + k^2 w \overline{\psi_2}} \, \dd x 
\end{align}
and $\mathcal{B}: X(D)\times X(D) \to \mathbb{C}$ is given by
\begin{align}\label{seqformB}
\mathcal{B}\big((v,w) \, ; (\psi_1, \psi_2) \big)  = \int_D v\overline{\psi_1}\, \dd x.
\end{align}
Notice that the $L^2(\exD)$ term in the variational formulation is not compact since $\exD$ is unbounded see for e.g. \cite{adams-embedding}. By the Riesz representation theorem, there exist a bounded linear operator $\mathbb{A}_k : X(D) \to X(D) $ such that
\begin{align}
\mathcal{A}_k\big((v,w) \, ; (\psi_1, \psi_2) \big) &=  \big(\mathbb{A}_k(v,w) \, ; (\psi_1, \psi_2) \big)_{X(D)} \label{OpAkDef} 
\end{align}
and $\mathbb{B}: X(D) \to X(D)$ satisfying 
\begin{align}
\mathcal{B}\big((v,w) \, ; (\psi_1, \psi_2) \big) &=\big(\mathbb{B}(v,w) \, ; (\psi_1, \psi_2) \big)_{X(D)} \label{OpBDef}
\end{align}
where 
$$({\cdot \, ; \cdot})_{X(D)} := ({\cdot \, , \cdot})_{H^1(D)} + ({\cdot \, , \cdot})_{H^1(\exD)}$$ 
denotes the standard product space inner--product in $X(D)$ defined by \eqref{spacex}. Also, it is clear from the definitions of the sesquilinear forms that both $\mathbb{A}_k$ and $\mathbb{B}$ are self--adjoint for all $k>0$. 

Notice that $k$ is a transmission eigenvalue satisfying \eqref{tep1}--\eqref{tep2} if and only if $\mathbb{A}_k - k^2 \mathbb{B}$ has a non--trivial null space. To show the discreteness of the set of transmission eigenvalues in $\C_{+}$, we will apply the analytic Fredholm theorem (see Theorem 1.24, \cite{Cakoni-Colton-book}). To this end, we need to show that the operator $\mathbb{A}_k - k^2 \mathbb{B}$ is Fredholm and depends analytically on $k \in \C_{+}$. With this in mind we now give the following lemma. 

\begin{lemma}\label{fredholm+analytic}
The operator $\mathbb{A}_k - k^2 \mathbb{B}$ defined by \eqref{OpAkDef}--\eqref{OpBDef} is Fredholm of index zero and depends analytically on $k \in \C_{+}$.
\end{lemma}
\begin{proof}
First, we prove that the operator depends analytically on $k \in \C_{+}$. This is a direct consequence of the fact that the sesquilinear form $\mathcal{A}_k$ is weakly analytic with respect to $k$, which implies that it is strongly analytic and $k^2$ depends analytically on $k$.

To prove that the operator is Fredholm of index zero, we show that it can be written as the sum of a coercive operator and a compact operator. First, we show that the operator $\mathbb{B}$ is compact. Indeed, by \eqref{seqformB} we have that 
$$ \left| \big(\mathbb{B}(v,w) \, ; (\psi_1, \psi_2) \big)_{X(D)} \right| \leq \|v\|_{L^2(D)} \|(\psi_1, \psi_2) \|_{X(D)} $$ 
which implies that $\|\mathbb{B}(v, w)\|_{X(D)}\leq C \|v\|_{L^2(D)}$ and by the compact embedding of $H^1(D)$ into $L^2(D)$ we have the compactness.  

Now, we prove the coercivity of the operator $\mathbb{A}_k$ for all $k \in \C_{+}$. To this end, notice that by our assumption that $ \text{Re}(k^2)> 0$ and \eqref{seqformAk} we have that 
$$ \text{Re}\left[ \mathcal{A}_k\big((v,w) \, ; (v,w) \big) \right] = \|\nabla v\|^2_{L^2(D)} +\|\nabla w\|^2_{L^2(\exD)} + \text{Re}(k^2)\| w\|^2_{L^2(\exD)}.$$
To prove the claim, we proceed by contradiction and assume there exists a sequence $(v_n, w_n) \in X(D)$ for $n \in \N$ such that  $\|(v_n, w_n)\|_{X(D)}=1$ satisfying 
$$\frac{1}{n} \geq \left|\mathcal{A}_k\big((v_n, w_n) \, ; (v_n, w_n) \big) \right| \geq \|\nabla v_n\|^2_{L^2(D)} +\|\nabla w_n\|^2_{L^2(\exD)}+ \text{Re}(k^2)\| w_n\|^2_{L^2(\exD)} . $$ 
Since the sequence is bounded, there exists a subsequence that converges weakly to $(v,w)$ in $X(D)$. This yields that 
$$\nabla v = 0 \:\: \text{in}\:\: D \quad  \text{and} \quad  w = 0 \:\: \text{in}\:\: \exD. $$ 
It follows that $v$ is a constants. Notice that our assumption implies that $w_n$ converges strongly to zero.  We also have that the sequence $v_n$ is strongly convergent since $H^1(D)$ is compactly embedded in $L^2(D)$ and the gradient is strongly convergent to the zero vector. From the essential boundary condition $v+w=0$ on $\partial D$, we have that $v=0$ on $\partial D$ which implies that $v$ also vanishes in $D$. This contradicts the fact that $\|(v,w)\|_{X(D)}=1$ by the strong convergence. Which proves the coercivity and therefore proves the claim.
\end{proof}

From the above analysis, we have that $\mathbb{A}_k - k^2 \mathbb{B}$ is Fredholm of index zero and depends analytically on $k \in \C_{+}$. By the analytic Fredholm theorem we have that the operator $\mathbb{A}_k - k^2 \mathbb{B}$ is either not invertible for any in $k \in \C_{+}$ or is invertible for almost every $k \in \C_{+}$. This implies that we only need to show that the operator is invertible on a subset of $\C_{+}$ (or even just at one point). To this end, we now prove the following lemma.

\begin{lemma}\label{realTE}
If $k  \in \C_{+}$ is a transmission eigenvalue satisfying \eqref{tep1}--\eqref{tep2}, then $k$ is real--valued. 
\end{lemma}
\begin{proof}
Assume that $k  \in \C_{+}$ is a transmission eigenvalue with corresponding eigenfunctions $({v_k, w_k}) \in X(D)$ such that $\|({v_k, w_k}) \|_{X(D)} = 1$. Since the eigenfunction satisfy \eqref{TEeq}, we have that
$$\frac{1}{k^2} \left({\|\nabla v_k\|_{L^2(D)}^2 + \|\nabla w_k\|_{L^2(\exD)}^2} \right)= \|v_k\|_{L^2(D)}^2 - \|w_k\|_{L^2(\exD)}^2.$$
Taking the imaginary part gives 
$$\text{Im} \left({\frac{1}{k^2}} \right) \left({\|\nabla v_k\|_{L^2(D)}^2 + \|\nabla w_k\|_{L^2(\exD)}^2}\right) = 0.$$
Now, by way of contradiction we assume that the eigenvalue $k \in \C_{+}$ is given by $k = \alpha + \mathrm{i}\beta$ such that $\beta \neq 0$. The above equality implies that 
$$ {\|\nabla v_k\|_{L^2(D)}^2 + \|\nabla w_k\|_{L^2(\exD)}^2} = 0$$ 
since 
$$\text{Im} \left({\frac{1}{k^2}} \right) = -\frac{2 \alpha \beta}{|k|^4} \neq 0  \quad \text{and} \quad \alpha \neq 0  \quad \text{for all $k\in \C_+$}.$$
Therefore, we have that $\nabla v_k$ and $\nabla w_k$ vanish in $D$ and $\exD$, respectively. It follows that $v_k$ and $w_k$ are constants. Since $w_k$ satisfies the radiation condition, $w_k$ must be zero. From the boundary conditions, we have that $v_k = \partial_\nu{v_k}=0$ on $\partial D$ and thus $v_k$ is also zero by Green's representation formula since it satisfies the Helmholtz equation in $D$. This contradicts the fact that $\|({v_k, w_k}) \|_{X(D)} = 1$. So we conclude that $\beta=0$, which proves the claim. 
\end{proof}

The analysis in Lemma \ref{realTE} implies that the operator $\mathbb{A}_k - k^2 \mathbb{B}$ is injective for all $k \in \C_{+} \backslash \R$. Since the operator is Fredholm of index zero, we have that injectivity implies invertibility. By the analytic Fredholm theorem this gives that the set of transmission eigenvalues in $\C_+$ is at most discrete. With this we have proven the following result.

\begin{theorem}\label{discreteness}
The set of transmission eigenvalues $k \in \C_{+}$ satisfying \eqref{tep1}--\eqref{tep2} is at most a discrete subset of the real line.
\end{theorem}
\begin{proof}
This is a direct consequence of the above discussion. 
\end{proof}

With Theorem \ref{discreteness}, we have proven that the set of transmission eigenvalues $k>0$ is at most discrete. This result also has applications to the inverse shape problem for recovering the scatterer $D$. In \cite{LSM-BHclamped} the linear sampling method was studied associated with the direct scattering problem \eqref{biharmonic}--\eqref{cbc}. As for other inverse scattering problems, it is well--known that the linear sampling method is not valid for transmission eigenvalues. With this result we now have that linear sampling method for recovering a clamped obstacle is valid for almost any wave number $k>0$. \\ 


\noindent{\bf Transmission Eigenvalues for the Unit Disk:} Before we proceed to the next section, notice that the calculation in Lemma \ref{realTE} seems to formally imply that the transmission eigenvalues are either real--valued or purely imaginary. To this end, we consider an analytical example to see if we can establish the existence of purely imaginary transmission eigenvalues. Therefore, we let $D = B_1$, where $B_1$ is the disk of radius 1 centered at $0$. Then, we have that
\begin{align*}
\Delta v + k^2 v = 0 \quad \text{in  } B_1 \quad &\text{and}\quad \Delta w - k^2 w  = 0 \quad \text{in  } \,\,  \mathbb{R}^2 \setminus \ov{B_1},\\
({v+w})|_{\partial B_1} = 0, \quad \pdr ({v+w})|_{\partial B_1} = 0  \quad &\text{and } \quad   
\lim_{r \to \infty} \sqrt{r}(\pdr{w} - \text{i}kw) = 0.
\end{align*}
Using the separation of variable, one can obtain
\begin{align*}
v (r, \theta)= \sum\limits_{| {\ell} | =0}^{\infty} v^{({\ell})} J_{\ell}(kr)\text{e}^{\text{i} \ell \theta} \quad \text{and} \quad w(r, \theta) = \sum\limits_{| {\ell}  | =0}^{\infty} w^{({\ell})} H_{\ell}^{(1)}({\mathrm{i}k}r)\text{e}^{\text{i} \ell \theta},
\end{align*}
where $J_{\ell}$ is the Bessel function of first kind of order ${\ell}$ and $H_{\ell}^{(1)}$ is the Hankel function of first kind of order ${\ell}$. Notice that the eigenfunction $w$ satisfies the radiation condition as $r \to \infty$. Here, $v^{({\ell})}$ and $w^{({\ell})}$ are unknown coefficients to be determined by the boundary conditions. Therefore, applying the boundary conditions at $r=1$, we derive the linear systems  
\begin{align*}
\begin{pmatrix}
J_{\ell}(k) & H_{\ell}^{(1)}({\mathrm{i}k}) \\
kJ'_{\ell}(k) & {\mathrm{i}k} {H_{\ell}^{(1)}}'({\mathrm{i}k})
\end{pmatrix}
\begin{pmatrix}
v^{(\ell)}  \\ w^{(\ell)} 
\end{pmatrix}
= \begin{pmatrix} 0 \\ 0 \end{pmatrix} \quad \text{for all } \quad \ell \in \Z.
\end{align*}
Then, we have that $k$ is a transmission eigenvalue if and only if $|f_{\ell} (k)|=0$ for any ${\ell} \in \Z$, where the determinant function $f_{\ell}(k)$ is given by
\begin{align}\label{det}
f_{\ell}(k) = \det{\begin{pmatrix}
J_{\ell}(k) & H_{\ell}^{(1)}(\mathrm{i}k) \\
kJ'_{\ell}(k) & \mathrm{i}k {H_{\ell}^{(1)}}'(\mathrm{i}k)
\end{pmatrix}
} = \mathrm{i}k J_{\ell}(k){H_{\ell}^{(1)}}'(\mathrm{i}k) - k H_{\ell}^{(1)}(\mathrm{i}k)J'_{\ell}(k)
\end{align}
From the recursion relations for the derivatives of $J_{\ell}$ and $H_{\ell}^{(1)}$, we have that
\begin{align}
f_{\ell}(k) = \frac{\mathrm{i}k}{2} J_{\ell}(k) \paren{H_{{\ell}-1}^{(1)}(\mathrm{i}k) - H_{{\ell}+1}^{(1)}(\mathrm{i}k)  } - \frac{k}{2} H_{\ell}^{(1)}(\mathrm{i}k) \paren{J_{{\ell}-1}(k) - J_{{\ell}+1}(k)}.\label{detfn}
\end{align}
In Figure \ref{fig_det}, we present plots of $|f_{\ell}(\mathrm{i}k)|$ for ${\ell} = 0, 1,2,3,4,5$. These plots seem to indicate that $|f_{\ell}(\mathrm{i}k)|>0$ and grows exponentially.  This would indicate that no pure imaginary transmission eigenvalue exists for the unit disk. 

\begin{figure}[H]
\centering
        \includegraphics[width=0.48\linewidth]{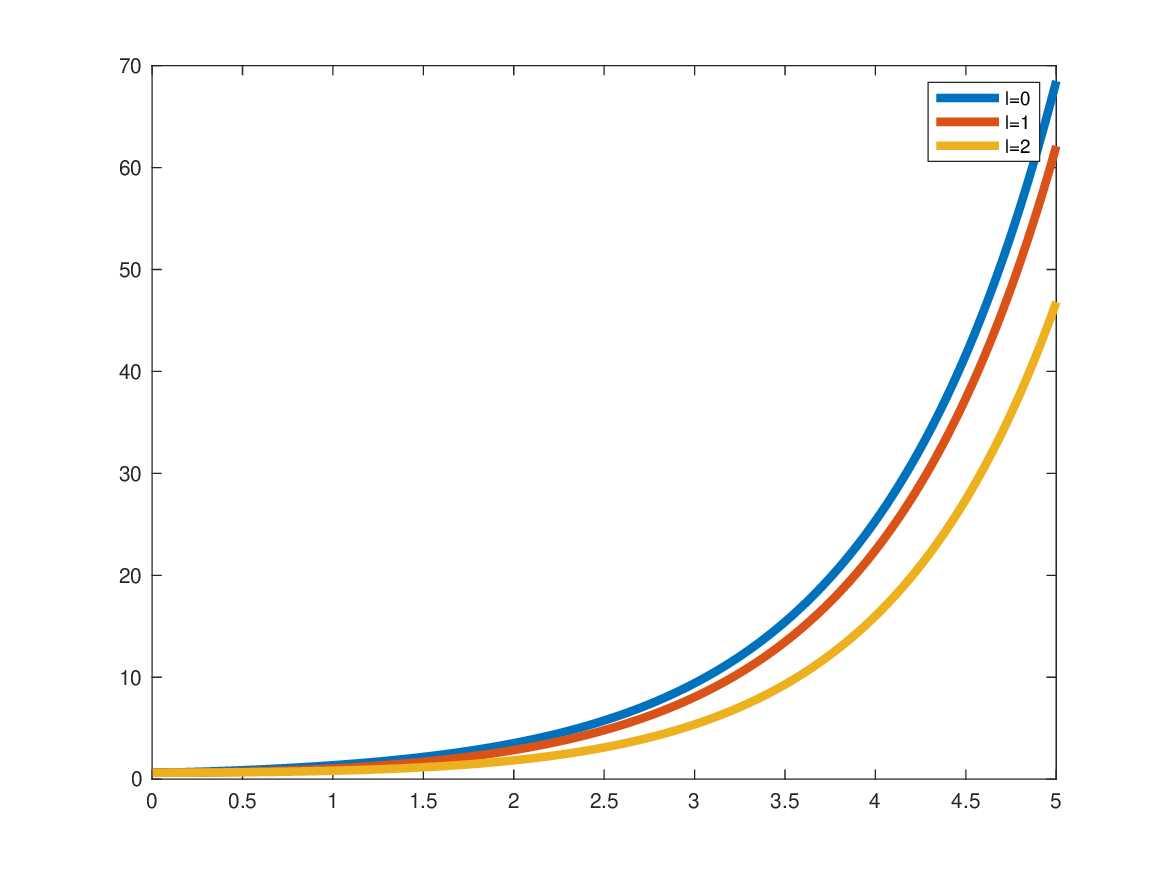}     
       \includegraphics[width=0.48\linewidth]{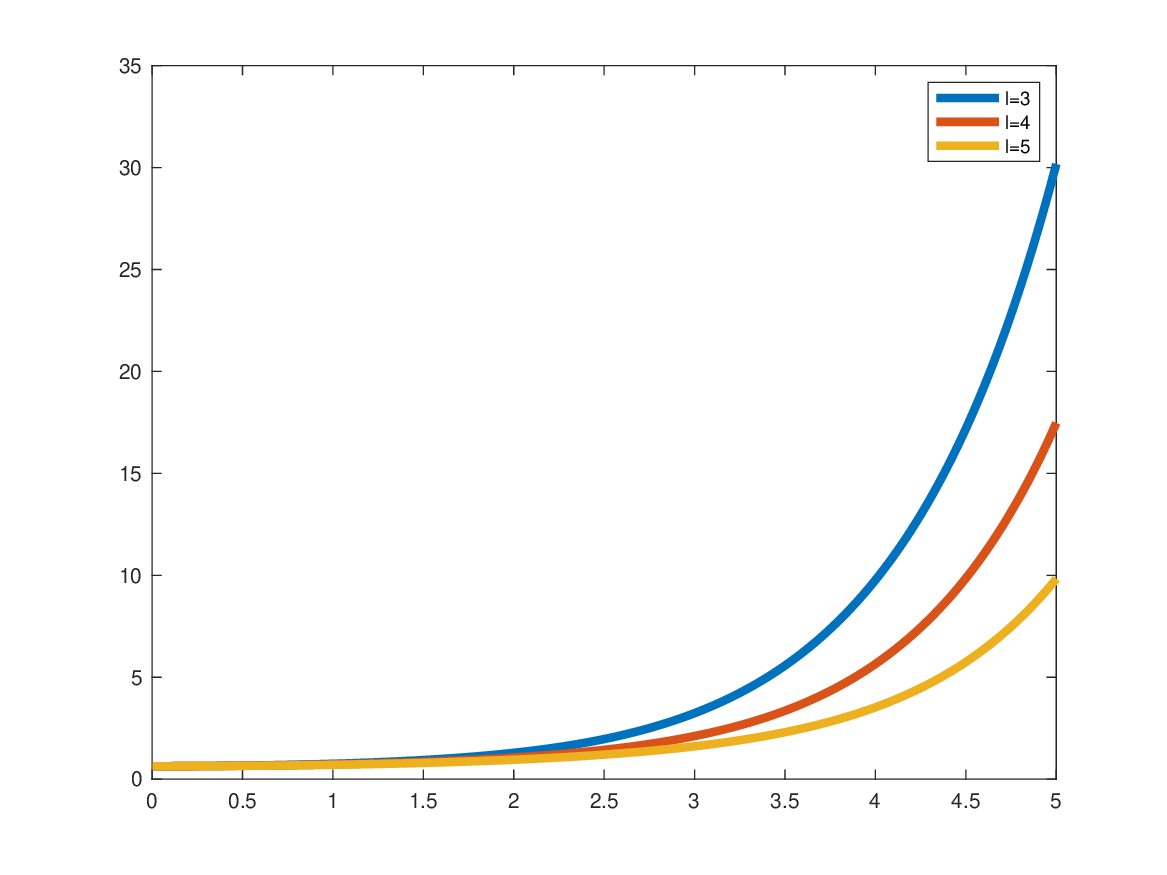} 
    \caption{The graph of the function $|f_{\ell} (\mathrm{i}k)|$ with $k \in [0,5]$ for ${\ell}=0,1,2$ on the left and ${\ell}=3,4,5$ on the right as defined in \eqref{detfn}.}
    \label{fig_det}
\end{figure}

In our study, we considered  $|f_{\ell}(\mathrm{i}k)|$ for many $\ell$ and were unable to detect any purely imaginary roots. This would indicate that no pure imaginary transmission eigenvalue exists for the unit disk. Notice that the analysis in this section seems to point to the fact that the transmission eigenvalues in the complex plane are real--valued.

\section{Determination of Eigenvalues from Far Field Data}
In this section, we demonstrate the determination of the transmission eigenvalues from the far field data. The determination of the eigenvalues from the scattering data was first established in \cite{cchlsm} via the linear sampling method. The work in \cite{cchlsm} focused on the acoustic scattering problem for a sound soft scatterer and an isotropic scatterer. This analysis exploits the fact that the aforementioned linear sampling method is not valid at the transmission eigenvalues to recover them. Also, in \cite{armin} the so--called inside--outside duality method was established. This connects the determination of the transmission eigenvalues to the factorization of the far field operator. The inside--outside duality method has been studied for multiple transmission eigenvalue problems \cite{inout-modified,armin2,inout-iso}. Here, we will prove that the transmission eigenvalues $k>0$ can be recovered from the far field operator. This result will help in the applicability of the inverse spectral problem, i.e. recovering information about the scatterer $D$ from the knowledge of the transmission eigenvalues. The inverse spectral problem for other types of scatterers has been studied in \cite{ap-invTE,bck-invTE,homogeniz-TE,gp-invTE} to name a few. 

Now, recall that the far field operator associated with our direct scattering problem \eqref{vhvmeq1}--\eqref{vhvmeq2} is defined in \eqref{ffop}. We let  $\phi_z \coloneqq \text{e}^{-{\mathrm{i}k} \hat{x}\cdot z}$ represent the far field pattern of the fundamental solution $\Phi_k (\cdot \, , z)$ to the Helmholtz equation given in \eqref{fundsol}. We assume that for a fixed $z \in D$ that there exists $g_z^\varepsilon \in L^2(\mathbb{S}^1)$ satisfying 
\begin{align}\label{fgphi}
\|Fg_z^\varepsilon - \phi_z\|_{L^2(\s^1)} \to 0 \quad \text{as } \quad \varepsilon \to 0.
\end{align}
Using the fact that  $F g_z^\varepsilon = u^\infty_{H,g_z^\varepsilon}$, where $u_{H,g_z^\varepsilon}$ is the  Helmholtz part of the scattered field with incident field given by $v_{g_z^\varepsilon}$, we obtain from \eqref{fgphi} that
$$u^\infty_{H,g_z^\varepsilon} \to  \phi_z \quad \text{ as } \quad \varepsilon \to 0.$$
If we assume that $\|g_z^\varepsilon\|_{L^2(\mathbb{S}^1)}$ is bounded as $\varepsilon \to 0$, then by the well--posedness of \eqref{vhvmeq1}--\eqref{vhvmeq2} we have that the corresponding $u_{H,g_z^\varepsilon}$ and $u_{M,g_z^\varepsilon}$ satisfying \eqref{vhvmeq1}--\eqref{vhvmeq2} with $u^{\text{inc}}=v_{g_z^\varepsilon}$ are weakly convergent. Therefore, we let 
$$u_{M,g_z^\varepsilon} \weakc  w \,\,\,  \text{in $H^1(\exD)$} \quad \text{and} \quad v_{g_z^\varepsilon} \weakc v  \,\,\, \text{in $H^1(D)$}  \quad \text{as} \quad \varepsilon \to 0.$$ 
Notice that by Rellich's Lemma we have that 
$$u_{H,g_z^\varepsilon} \weakc  \Phi_k (\cdot \, , z)  \,\,\,  \text{in $H^1(B_R \setminus \overline{D})$}  \quad \text{as} \quad \varepsilon \to 0$$ 
by appealing to the limit \eqref{fgphi} and the fact that  $\Phi_k (\cdot \, , z)$ is a solution to the Helmholtz equation in $\exD$ for all $z \in D$. Here, the Helmholtz part of the scattered field does not decay fast enough to be in the space $H^1(\exD)$. Since the weak limits satisfy \eqref{vhvmeq1}--\eqref{vhvmeq2}, we have that 
\begin{align*}
\Delta v + k^2 v = 0 \quad \text{in } D \quad &\text{and } \quad   \Delta w - k^2 w = 0 \quad \text{in } \exD, \\
(v+w) = -\Phi_k(\cdot \, , z) , \quad \dnu(v+w) = -\dnu{\Phi_k(\cdot \, , z)} \quad \text{on } \partial D \quad &\text{and } \quad  
\lim_{r \to \infty} \sqrt{r}(\pdr{w} - \text{i}kw) = 0.
\end{align*}

Now, we wish to establish a variational formulation for the above boundary value problem. 
By appealing to Green's first identity, we have that 
$$\int_D {\nabla v \cdot \overline{\psi_1} - k^2 v \overline{\psi_1}}  \, \dd x 
+ \int_{\exD} {\nabla w \cdot \nabla \overline{\psi_2} + k^2 w \overline{\psi_2}}  \, \dd x\\
=  - \int_{\partial D} \dnu{\Phi_k(\cdot \, , z)} \overline{\psi_1}  \, \dd s(x)$$
for any $(\psi_1, \psi_2 ) \in X(D)$. Since we here consider $k>0$, we assume that $w \in H^1(\exD)$ as in the previous section. Again, due to the fact that $k>0$ we have that $w$ and $\pdr{w}$ decay exponentially, which satisfies the radiation condition. We need to impose the essential boundary condition. To this end, for any $z \in D$ we make the ansatz that there is a $v_0 \in H^1(D)$ such that $v = v_0 -\Phi_z$ where $\Phi_z \in H^1(D)$ satisfies $\Phi_z = \Phi_k(\cdot, z)$ on $\partial D$ (i.e. $\Phi_z$ is a lifting function for the essential boundary condition). Then, the pair $(v_0, w) \in X(D)$ and the above variational identity can be written as 
\begin{multline}\label{TEcont}
\int_D {\nabla v_0 \cdot \overline{\psi_1} - k^2 v_0 \overline{\psi_1}} \, \dd x 
+ \int_{\exD} {\nabla w \cdot \nabla \overline{\psi_2} + k^2 w \overline{\psi_2}}  \,\dd x\\
= \int_D {\nabla \Phi_z \cdot \nabla \overline{\psi_1} - k^2 \Phi_z \overline{\psi_1} } \,\dd x - \int_{\partial D} \dnu{\Phi_k(\cdot \, , z)} \overline{\psi_1}  \, \dd s(x).
\end{multline}
This implies that \eqref{TEcont} can be written as: find $(v_0, w) \in X(D)$ such that 
$$\mathcal{A}_k \big((v_0 ,w) \, ; (\psi_1, \psi_2) \big) - k^2 \mathcal{B}\big((v_0,w) \, ; (\psi_1, \psi_2) \big) = L \big( (\psi_1, \psi_2) \big)\quad \text{for all} \quad (\psi_1, \psi_2 ) \in X(D).$$
Here, the sesquilinear form $\mathcal{A}_k$ and $\mathcal{B}: X(D)\times X(D) \to \mathbb{C}$ are given by \eqref{seqformAk} and \eqref{seqformB}, respectively. The bounded conjugate linear functional ${L}:  X(D) \to \mathbb{C}$ is given by 
\begin{align}\label{funcLdef}
L \big( (\psi_1, \psi_2) \big) = \int_D {\nabla \Phi_z \cdot \nabla \overline{\psi_1} - k^2 \Phi_z \overline{\psi_1} }  \, \dd x - \int_{\partial D} \dnu{\Phi_k(\cdot \, , z)} \overline{\psi_1}  \, \dd s(x).
\end{align}
From Theorem \ref{fredholm+analytic}, we have that \eqref{TEcont} is Fredholm of index zero. With this, we are now ready to prove the main result of this section. 

\begin{theorem}\label{determineEig}
If $k >0 $ is a transmission eigenvalue satisfying \eqref{tep1}--\eqref{tep2} and $g_z^\varepsilon$ satisfy \eqref{fgphi}. Then, for $z\in D$, except for possibly a nowhere dense set,  $\|{g_z^\varepsilon}\|_{L^2(\mathbb{S}^1)}$ cannot be bounded as $\varepsilon \to 0$.
\end{theorem}
\begin{proof}
To prove the claim, let the wave number $k>0$ correspond to a transmission eigenvalue. We assume to the contrary that $\|g_z^\varepsilon\|_{L^2(\s^1)}$ is bounded for all $z \in \mathcal{M}$ such that $\mathcal{M} \subset D$ of positive measure. By the previous discussion, this implies \eqref{TEcont} has a solution. Due to the fact that \eqref{TEcont} is Fredholm of index zero and $k$ is a transmission eigenvalue, the Fredholm alternative implies that the corresponding eigenfunctions $({v_k, w_k}) \in X(D)$ such that $\|({v_k, w_k}) \|_{X(D)} = 1$ must satisfy 
$$L \big(({v_k, w_k}) \big) = \int_D {\nabla \Phi_z \cdot \nabla \overline{v_k} - k^2 \Phi_z \overline{v_k} }  \, \dd x - \int_{\partial D} \dnu{\Phi_k(\cdot \, , z)} \overline{v_k}  \, \dd s(x) = 0 \quad \text{ for all $z \in \mathcal{M}$}.$$
By Green's first identity we have that 
\begin{align*}
0 &=\int_D {\nabla \Phi_z \cdot \nabla \overline{v_k} - k^2 \Phi_z \overline{v_k} }  \, \dd x - \int_{\partial D} \dnu{\Phi_k(\cdot \, , z)} \overline{v_k}  \, \dd s(x) \\
  &= \int_{\partial D} \dnu{\overline{v_k}}  \Phi_k(\cdot \, , z) -  \dnu{\Phi_k(\cdot \, , z)} \overline{v_k}  \, \dd s(x) \quad \text{ for all $z \in \mathcal{M}$},
\end{align*}
where we have used that $v_k$ solves the Helmholtz equation in $D$ as well as the fact that the lifting function satisfies $\Phi_z = \Phi_k(\cdot, z)$ on $\partial D$. By Green's representation formula, we have that 
\begin{align*}
\int_{\partial D} \dnu{\overline{v_k}}  \Phi_k(\cdot \, , z) -  \dnu{\Phi_k(\cdot \, , z)} \overline{v_k}  \, \dd s(x) = 0, \quad \text{which implies that} \quad \overline{v_k} (z) = 0  \quad \text{ for all $z \in \mathcal{M}$}. 
\end{align*}
By unique continuation, we obtain that $v_k = 0$ in $D$ by the fact that $\mathcal{M}$ has positive measure. Since the eigenfunctions satisfy \eqref{TEeq}, we have that 
$$0= \|\nabla w_k \|^2_{L^2(\exD)^2} +k^2  \| w_k\|_{L^2(\exD)}^2,$$
which implies that $w_k = 0$ in $\exD$ since $k>0$. This contradicts the fact that $\|({v_k, w_k}) \|_{X(D)} = 1$, proving the claim. 
\end{proof}

By Theorem \ref{determineEig}, we have that the transmission eigenvalues can be recovered from the far field data. This result implies that the transmission eigenvalues can be `computed' provided that one knows the location of the scatterer. Studying the inverse spectral problem for this eigenvalue problem  could offer further valuable insights to biharmonic scattering. For the analogous, acoustic scattering problem the associated eigenvalues are the Dirichlet eigenvalue for the scatterer $D$. It is known that the Dirichlet eigenvalue are monotone with respect to the measure of $D$.  Therefore, the Dirichlet eigenvalues can be used to  determine the measure of the scatterer. Studying the dependance of the transmission eigenvalues on the scatterer $D$ is also an interesting analytical problem. \\

\noindent{\bf Analytic Example for Determining the Eigenvalues:} Now we consider validating Theorem \ref{determineEig} via an analytical example. A similar example was given in \cite{te-cbc2} for a different transmission eigenvalue problem. This is done be deriving a series expansion of the far field operator. Then we can solve the far field equation
$$ F g_z = \phi_z \quad \text{for a given } \quad z \in D$$
as an approximation to the limit \eqref{fgphi}, where $ \phi_z = \text{e}^{-{\mathrm{i}k} \hat{x}\cdot z}$. For our example, we will again assume that the scatterer is $D=B_1$. Therefore, the direct scattering problem \eqref{vhvmeq1}--\eqref{vhvmeq2} is given by 
\begin{align*}
\Delta u_H + k^2 u_H = 0  \quad &\text{and}\quad \Delta u_M - k^2 u_M  = 0 \quad \text{in  } \,\, \mathbb{R}^2 \setminus \ov{B_1},\\
({u_H+u_M})|_{\partial B_1} = -u^{\text{inc}} \quad &\text{and } \quad\quad \pdr ({u_H+u_M})|_{\partial B_1} = - \pdr u^{\text{inc}}      
\end{align*}
along with the radiation condition \eqref{SRC1}. Recall that the incident field is given by $u^{\text{inc}} =  \text{e}^{\text{i}kx\cdot d}$ for some fixed $d \in \mathbb{S}^1$. 

Just as in the previous section, we will use separation of variables. Since $u_H$ satisfies the Helmholtz equation and $u_M$ satisfies the modified Helmholtz equation we have that 
\begin{align*}
u_H (r, \theta) = \sum\limits_{| \ell | =0}^{\infty} u_H^{(\ell)} H_{\ell}^{(1)}(kr)\text{e}^{\text{i} \ell \theta} \quad \text{and} \quad u_M  (r,\theta)  = \sum\limits_{| \ell | =0}^{\infty} u_M^{(\ell)} H_{\ell}^{(1)}({\mathrm{i}k}r)\text{e}^{\text{i} \ell \theta}.
\end{align*}
Again, $J_{\ell}$ is the Bessel function of first kind of order $\ell$ and $H_{\ell}^{(1)}$ is the Hankel function of first kind of order $\ell$. Notice that the both functions satisfy the radiation condition as $r \to \infty$. As before, we will determine the coefficients $u_H^{(\ell)}$ and $u_M^{(\ell)}$ via the boundary condition at $r=1$. With this in mind, we will use the Jacobi--Anger expansion 
$$u^{\text{inc}} (r,\theta) = \sum\limits_{| \ell | =0}^{\infty} \text{i}^{\ell} J_{\ell}(kr)\text{e}^{\text{i}{\ell}(\theta - \phi)}$$
where $x= r \big(\cos(\theta),  \sin(\theta) \big)$ and $d = \big(\cos(\phi), \sin(\phi) \big)$. The boundary conditions then become the linear systems 
\begin{align*}
\begin{pmatrix}
H_{\ell}^{(1)}(kr) & H_{\ell}^{(1)}({\mathrm{i}k}) \\
k{H_{\ell}^{(1)}}'(kr) & {\mathrm{i}k} {H_{\ell}^{(1)}}' ({\mathrm{i}k})
\end{pmatrix}
\begin{pmatrix}
u_H^{({\ell})}  \\ u_M^{({\ell})}
\end{pmatrix}
=  -  \text{i}^{\ell} \text{e}^{-\text{i}{\ell} \phi} \begin{pmatrix}  J_{\ell}(k) \\ kJ_{\ell}'(k) \end{pmatrix} \quad \text{for all } \quad {\ell} \in \Z.
\end{align*}
Solving the system for $u_H^{({\ell})}$ using Cramer's rule, we obtain that $u_H^{({\ell})} =   \text{i}^{\ell}  \text{e}^{-\text{i}{\ell} \phi} \lambda_{\ell} (k) $ with
$$ \lambda_{\ell} (k) = -  \, \frac{ \mathrm{i}k J_{\ell}(k){H_{\ell}^{(1)}}'(\mathrm{i}k) - k H_{\ell}^{(1)}(\mathrm{i}k)J'_{\ell}(k) }{ \mathrm{i}k {H_{\ell}^{(1)}}(k){H_{\ell}^{(1)}}'(\mathrm{i}k) - k H_{\ell}^{(1)}(\mathrm{i}k){H_{\ell}^{(1)}}'(k) }$$
where we have made the dependance on $k$ explicit. Recall that $u_M$ decays exponentially fast so it does not contribute to the far field pattern. We now have that 
$$u_H^{\infty} (r,\theta) = \frac{ 4 }{ \text{i} }\sum\limits_{| {\ell} | =0}^{\infty}  \lambda_{\ell} (k)  \text{e}^{\text{i}{\ell}(\theta - \phi)} \quad \text{which implies that } \quad  Fg_z = \frac{ 4 }{ \text{i} }\sum\limits_{| {\ell}| =0}^{\infty} \lambda_{\ell} (k)  g_z^{({\ell})} \text{e}^{\text{i}{\ell}\theta } $$
where $g_z^{({\ell})}$ are the Fourier coefficients for $g_z$. By again using the Jacobi--Anger expansion for $\phi_z$ and solving for $g_z^{({\ell})}$ in the far field equation we obtain that 
$$g_z^{({\ell})} = (-1)^{\ell}  \frac{\text{i}^{{\ell}+1} }{4} \frac{J_{\ell}(k|z|)}{ \lambda_{\ell} (k) }  \text{e}^{-\text{i}{\ell}\theta_z }$$
where $z= |z| \big(\cos(\theta_z),  \sin(\theta_z) \big)$. 

From this we have that for some  ${\ell} \in \Z$ the function $\lambda_{\ell} (k) \to 0$ as $k$ approaches a transmission eigenvalue by \eqref{det}. This implies that for some  ${\ell} \in \Z$ that the corresponding Fourier coefficient becomes unbounded. Since one of the Fourier coefficients becomes infinite at a transmission eigenvalue that implies that the solution to the far field equation becomes unbounded as $k$ approaches a transmission eigenvalue. This analytical example is a direct representation of Theorem \ref{determineEig}.

\section{Numerical Experiments}
In this section, we provide some numerical experiments for our transmission eigenvalue problem. To this end, we will now derive a system of boundary integral equations to compute the transmission eigenvalues. We would like to note that, since this eigenvalue problem is on an infinite domain, a finite element method could be computationally expensive (especially in three dimensions). The method used to compute the eigenvalues is similar to what has been done in \cite{cakonikress,te-2cbc,te-cbc2}. Recall that we want to solve 
$$\Delta w-k^2 w=0\,\, \,  \text{ in $\mathbb{R}^2\backslash \overline{D}$} \quad \text{and }\quad  \Delta v+k^2 v=0 \,\,\, \text{in $D$}$$ 
with the boundary conditions 
$$w+v=0   \quad \text{and }\quad \partial_\nu (w+v)=0  \,\,\, \text{on $\partial D$},$$
where  $w$ satisfies the Sommerfeld radiation condition. 
We will use boundary integral equations for solving the problem at hand. We make the ansatz
\begin{eqnarray}
 w(x)&=&\mathrm{SL}_{\mathrm{i}k}\varphi(x)\,,\quad x\in\mathbb{R}^2 \backslash \overline{D}\,,\label{start1}\\
 v(x)&=&\mathrm{SL}_{k}\psi(x)\,,\quad x\in D\,,
 \label{start2}
\end{eqnarray}
where the single--layer operator is defined by
\begin{eqnarray*}
\mathrm{SL}_{\tau}\phi(x)&=&\int_{\partial D}\Phi_\tau (x,y)\phi(y)\,\mathrm{d}s(y)\,,\quad  x\notin \partial D\,,
\end{eqnarray*}
with $\Phi_\tau(x,y)$ the radiating fundamental solution of the Helmholtz equation for the wave number $\tau$ in two dimensions. Now, we let $x$ approach the boundary in (\ref{start1}) and (\ref{start2}). With the jump conditions (see \cite[Theorem 3.1]{coltonkress}), we obtain
\begin{eqnarray}
 w(x)&=&\mathrm{S}_{\mathrm{i}k}\varphi(x)\,,\quad x\in\partial D\,,\label{start1a}\\
 v(x)&=&\mathrm{S}_{k}\psi(x)\,,\quad x\in \partial D\,,
 \label{start2a}
\end{eqnarray}
where the single--layer boundary integral operator is defined by
\begin{eqnarray*}
    \mathrm{S}_{\tau}\phi(x)&=&\int_{\partial D}\Phi_\tau(x,y)\phi(y)\,\mathrm{d}s\,,\quad x\in \partial D\,.
\end{eqnarray*}
Taking the normal derivative and letting $x$ approach the boundary along with the jump conditions (\cite[Theorem 3.1]{coltonkress}), we obtain
\begin{eqnarray}
 \partial_{\nu(x)} w(x)&=&\left(-\frac{1}{2}I+\mathrm{D}^\top_{\mathrm{i}k}\right)\varphi(x)\,,\quad x\in\partial D\,,\label{start1b}\\
 \partial_{\nu(x)} v(x)&=&\left(\frac{1}{2}I+\mathrm{D}^\top_{k}\right)\psi(x)\,,\quad x\in \partial D\,,
 \label{start2b}
\end{eqnarray}
where the normal derivative of the single--layer boundary integral operator is defined by
\begin{eqnarray*}
    \mathrm{D}^\top_{\tau}\phi(x)&=&\int_{\partial D}\partial_{\nu(x)}\Phi_\tau(x,y)\phi(y)\,\mathrm{d}s\,,\qquad x\in \partial D
\end{eqnarray*}
for the wave number $\tau$. We solve (\ref{start1a}) for $\varphi$ and plug this into (\ref{start1b}). Likewise, we solve (\ref{start2a}) for $\psi$ and insert this into (\ref{start2b}). This yields
\begin{eqnarray}
 \partial_{\nu(x)} w(x)&=&\left(-\frac{1}{2}I+\mathrm{D}^\top_{\mathrm{i}k}\right)\mathrm{S}_{\mathrm{i}k}^{-1}w(x)\,,\quad x\in\partial D\,,\label{start3a}\\
 \partial_{\nu(x)} v(x)&=&\left(\frac{1}{2}I+\mathrm{D}^\top_{k}\right)\mathrm{S}_{k}^{-1}v(x)\,,\quad x\in \partial D\,.\label{start3b}
\end{eqnarray}
Using the boundary condition $w=-v$ and $\partial_\nu w+\partial_\nu v=0$ within (\ref{start3a}) and (\ref{start3b}) ultimately gives
\begin{eqnarray*}
 \left(\left(-\frac{1}{2}I+\mathrm{D}^\top_{\mathrm{i}k}\right)\mathrm{S}_{\mathrm{i}k}^{-1}-\left(\frac{1}{2}I+\mathrm{D}^\top_{k}\right)\mathrm{S}_{k}^{-1}\right)v(x)=0\,,\quad x\in \partial D\,.
\end{eqnarray*}
After discretization with the boundary element collocation method, yields a non-linear eigenvalue problem which is solved with the Beyn algorithm \cite{beyn}. The solution is an approximation for $k$ and an approximation of the eigenfunction $v$ on the boundary. Note that we then can compute the density $\varphi$ in (\ref{start1a}) using $w=-v$ and with this, the solution of $w$ within $\mathbb{R}^2\backslash \overline{D}$ using (\ref{start1}). Likewise, we can compute the density $\psi$ using (\ref{start1b}) and with this, the solution $v$ within $D$ using (\ref{start2}). 

\subsection{The unit disk}
To test our numerical algorithm, we compute the first four transmission eigenvalues for the unit disk with the boundary element collocation method using 120 collocation nodes and compare the results with the zeros of the determinant
\begin{eqnarray*}
    \det\begin{pmatrix}
        H_{\ell}^{(1)}(\mathrm{i}k) & J_{\ell}(k)\\
        H_{\ell}^{(1)'}(\mathrm{i}k) & J_{\ell}'(k)
    \end{pmatrix}=0\,,\quad \ell\in \mathbb{N} \cup \{ 0\}
\end{eqnarray*}
which was derived in Section \ref{TEdiscrete}. Here, we again let $J_{\ell}$ and $H_{\ell}^{(1)}$ denote the first kind Bessel function and the first kind Hankel function of order $\ell \in \mathbb{N}\cup \{ 0\}$, respectively. Note that we only consider $\ell \in \mathbb{N} \cup \{ 0\}$ due to the fact that $J_{-\ell} = (-1)^\ell J_{\ell}$ and $H_{-\ell}^{(1)} = (-1)^\ell H_{\ell}^{(1)}$. We obtain 
five digits accuracy as we can see in Table \ref{comparison}.
\begin{table}[!ht]
\centering
 \begin{tabular}{r|c|c}
 TE & determinant & BEM\\
\hline
$k_1$ & 1.6146349995639885158 $(\ell=0)$ & 1.61464\\
$k_2$ & 3.0516335028155405705 $(\ell=1)$ & 3.05164\\
$k_3$ & 3.0516335028155405705 $(\ell=1)$ & 3.05164\\
$k_4$ & 4.3645169097857215923 $(\ell=2)$ & 4.36453\\
$k_5$ & 4.3645169097857215923 $(\ell=2)$ & 4.36453\\
\hline
 \end{tabular}
 \caption{\label{comparison}The first five transmission eigenvalues (TE) for a unit disk via the determinant with 20 digits accuracy using various $\ell \in\mathbb{N}\cup \{ 0\}$ and the boundary element collocation method (BEM) with 120 collocation nodes.} \label{TEcircleTable}
\end{table}
Of course, we can get better results by increasing the number of collocation nodes. In Figure \ref{eigenfunctions1} and \ref{eigenfunctions2}, we show the real part of the first four eigenfunctions corresponding to $k_1\approx 1.61464$, $k_2\approx 3.05164$, $k_3\approx 3.05164$, and $k_4\approx 4.36453$ for the unit disk.

\begin{figure}[H]
    \centering
    \includegraphics[width=0.4\linewidth]{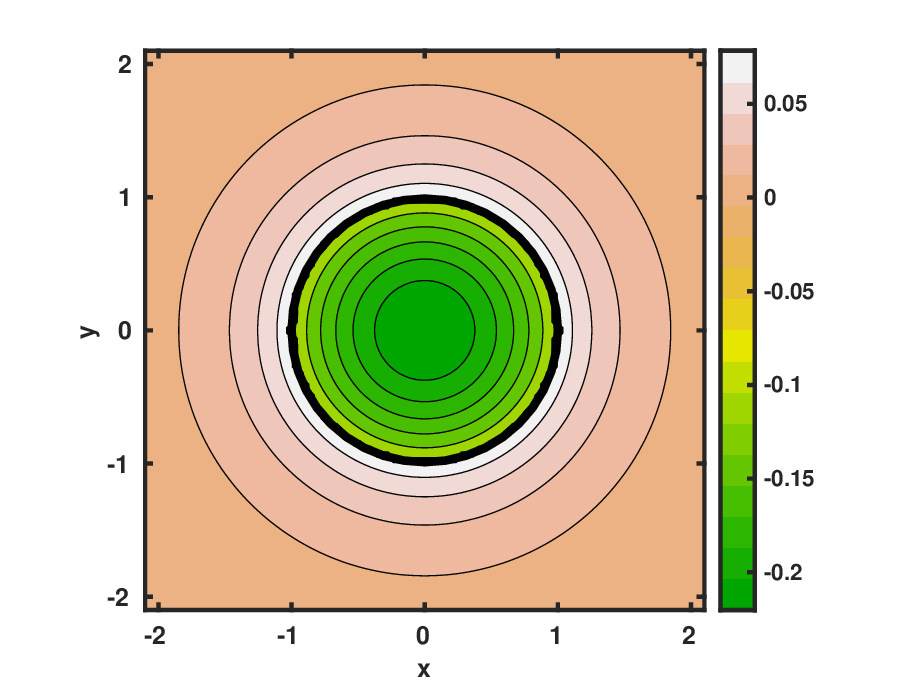}
    \includegraphics[width=0.4\linewidth]{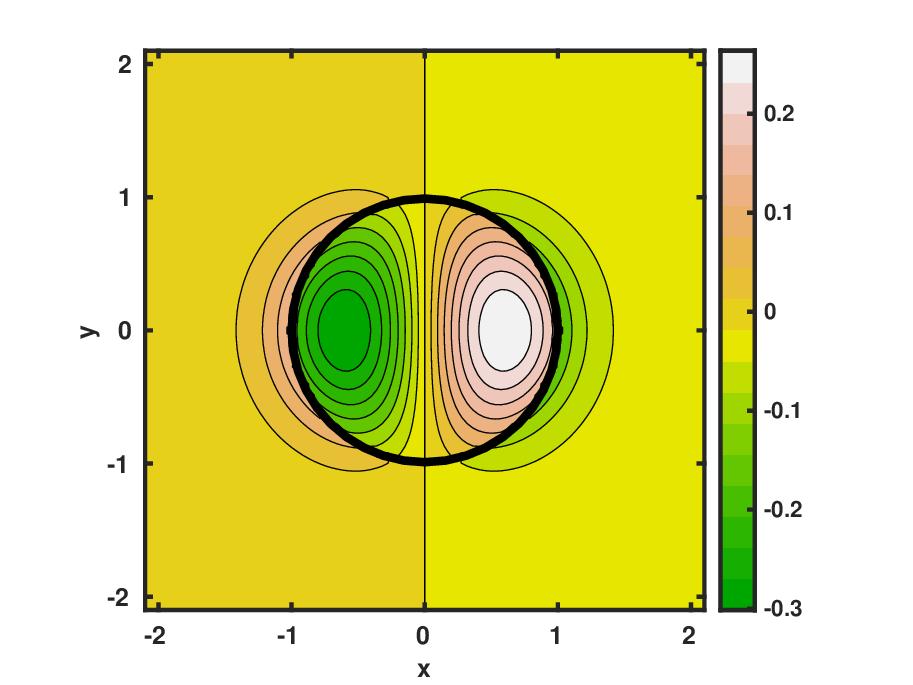}
    \caption{Plots of the real part of the eigenfunctions $w$ within $\mathbb{R}^2\backslash \overline{D}$ and $v$ within $D$ for the unit disk corresponding to $k_1\approx 1.61464$(left) and $k_2\approx 3.05164$(right).}
    \label{eigenfunctions1}
\end{figure}

\begin{figure}[H]
    \centering
    \includegraphics[width=0.4\linewidth]{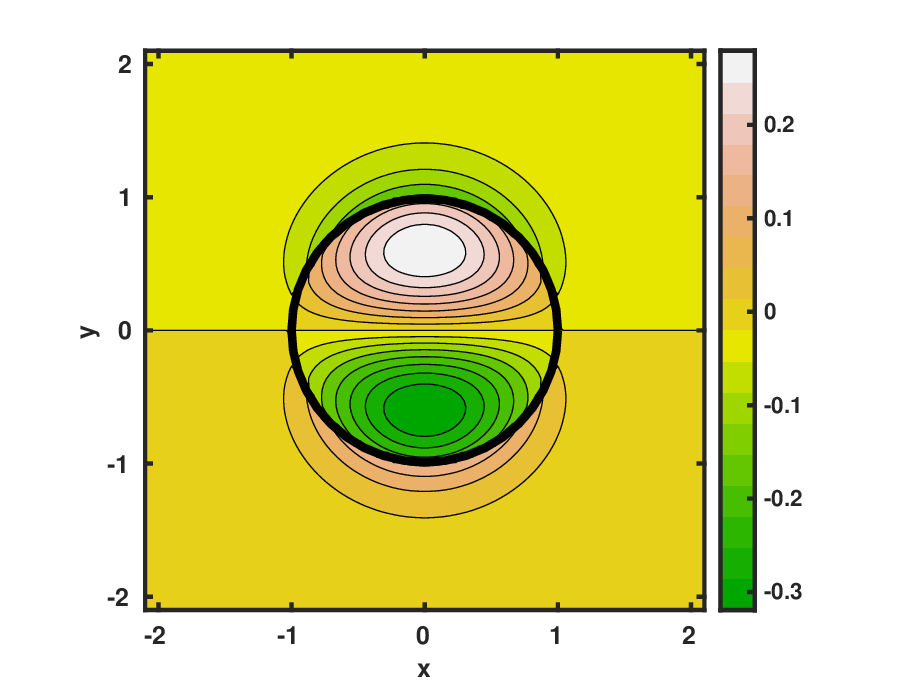}
    \includegraphics[width=0.4\linewidth]{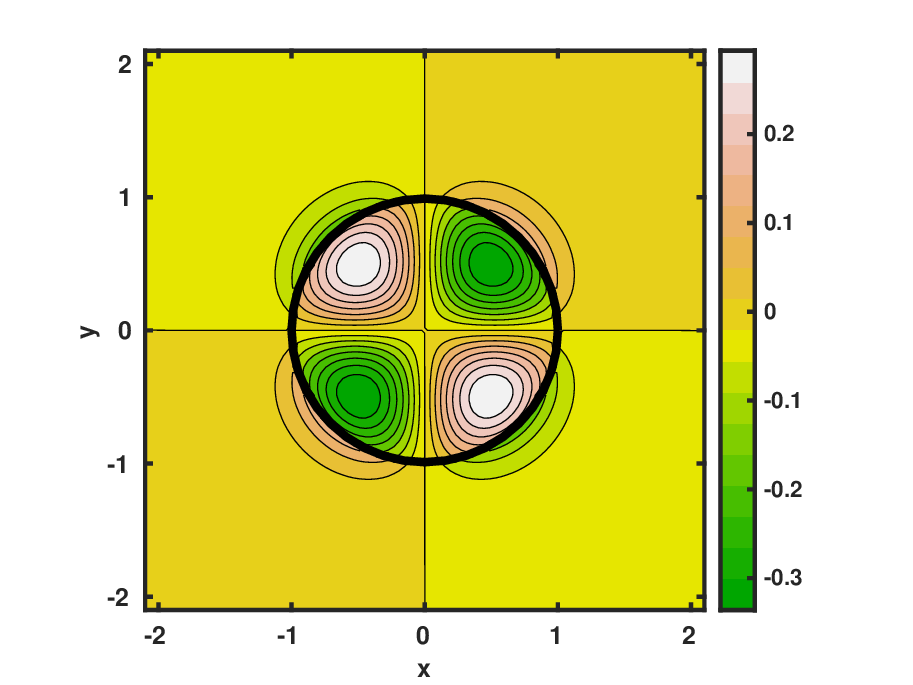}
    \caption{Plots of the real part of the eigenfunctions $w$ within $\mathbb{R}^2\backslash \overline{D}$ and $v$ within $D$ for the unit disk corresponding to $k_3\approx 3.05164$(left) and $k_4\approx 4.36453$(right).}
    \label{eigenfunctions2}
\end{figure}

\subsection{The ellipse}
Next, we compute the first four transmission eigenvalues for various ellipses. We let the major half--axis be $a$ and the minor half--axis be $b$.  We use for $(a,b)$ the parameter set $(1,1)$, $(1,0.9)$, $(1,0.8)$, $(1,0.7)$, $(1,0.6)$, and $(1,0.5)$. This gives that the boundary of the scatterer is given by 
$$\partial D = \big(\cos(t), b \sin(t)\big)^\top \,\,\text{ for $t\in [0,2\pi]$}$$ 
for varies values of $b$. As we observe in Table \ref{ellipse}, we have monotonicity of the first eigenvalue with respect to the area. Precisely, if we decrease the area, the first transmission eigenvalue increases. Note that the area of the chosen ellipses are given by $\pi\cdotp b$, where $b$ is chosen to be decreasing. Hence, the area is decreasing.
\begin{table}[!ht]
\centering
 \begin{tabular}{r|c|c|c|c|c|c}
 TE & $(1,1)$ & $(1,0.9)$ & $(1,0.8)$ & $(1,0.7)$ & $(1,0.6)$ & $(1,0.5)$\\
\hline
$k_1$ &  1.61464 & 1.70401& 1.81492 & 1.95646 & 2.14377 & 2.40418\\
$k_2$ &  3.05164 & 3.13673& 3.24382 & 3.38233 & 3.56787 & 3.82845\\
$k_3$ &  3.05164 & 3.30629& 3.62554 & 4.03737 & 4.58853 & 5.26231\\
$k_4$ &  4.36453 & 4.54060& 4.67571 & 4.82125 & 5.00599 & 5.36324\\
\hline
 \end{tabular}
 \caption{\label{ellipse}The first four transmission eigenvalues (TE) for various ellipses with half-axis $(a,b)$ using the boundary element collocation method (BEM) with 120 collocation nodes.}
\end{table}
Interestingly, we observe a monotonicity behavior with respect to all transmission eigenvalues. 
From the discussion in Section \ref{TEdiscrete}, we saw that there seems to be no purely imaginary transmission eigenvalues. For this numerical example, we note that we also tried to find pure complex transmission eigenvalues. However, none were found. We additionally show in Figure \ref{eigenfunctionselli1} and \ref{eigenfunctionselli2} the first four eigenfunctions corresponding to $k_1\approx 1.81492$, $k_2\approx 3.24382$, $k_3\approx 3.62554$, and $k_4\approx 4.67571$ for an ellipse $(1,0.8)$. We obtain similar results for the other ellipses.  

\begin{figure}[H]
    \centering
    \includegraphics[width=0.4\linewidth]{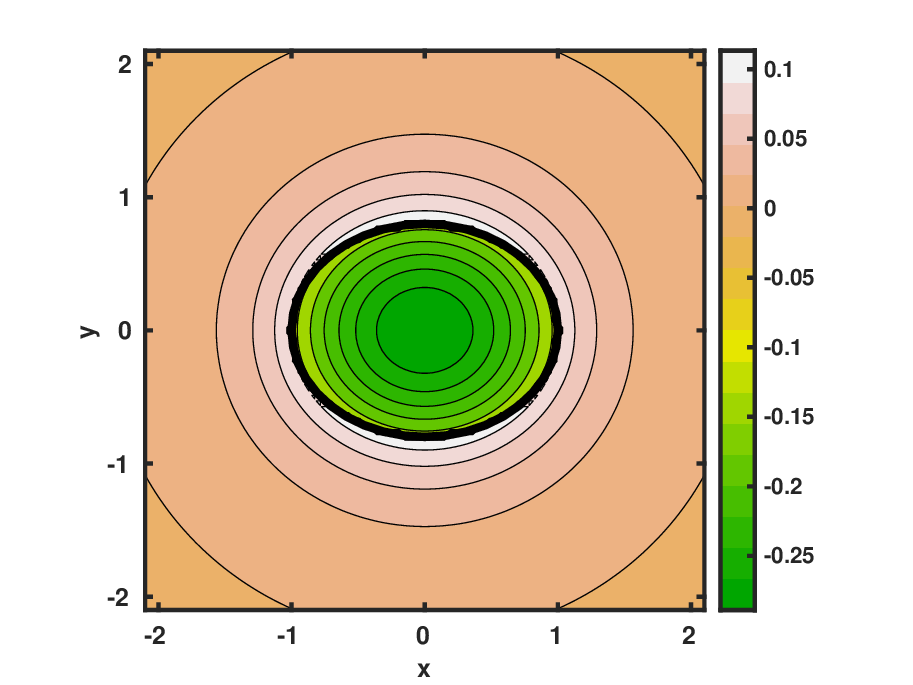}
    \includegraphics[width=0.4\linewidth]{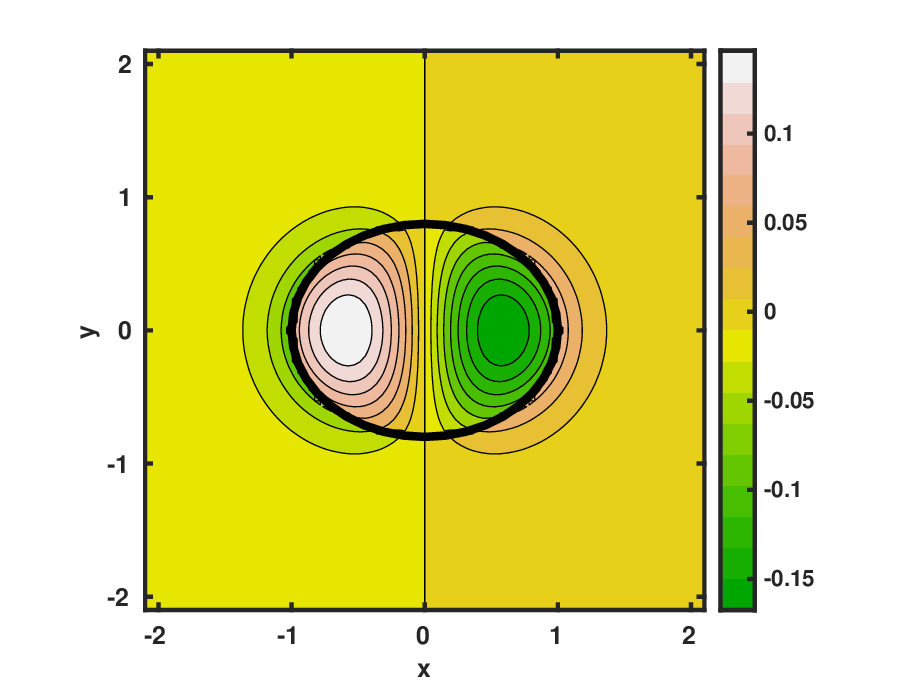}
    \caption{Plots of the real part of the eigenfunctions $w$ within $\mathbb{R}^2\backslash \overline{D}$ and $v$ within $D$ for the ellipse $(1,0.8)$ corresponding to $k_1\approx 1.81492$(left) and $k_2\approx 3.24382$(right).}
    \label{eigenfunctionselli1}
\end{figure}

\begin{figure}[H]
    \centering
    \includegraphics[width=0.4\linewidth]{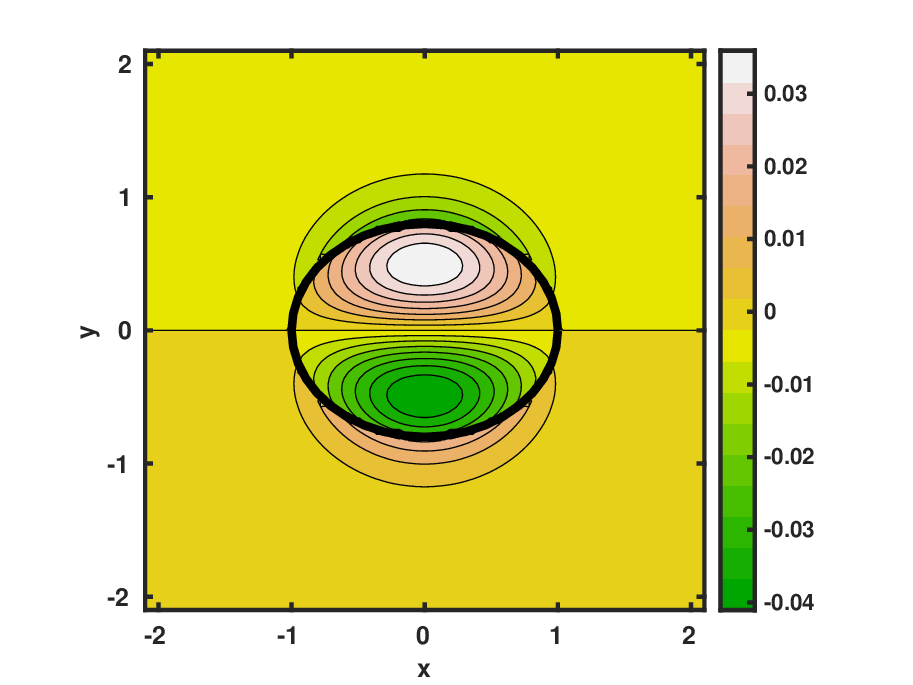}
    \includegraphics[width=0.4\linewidth]{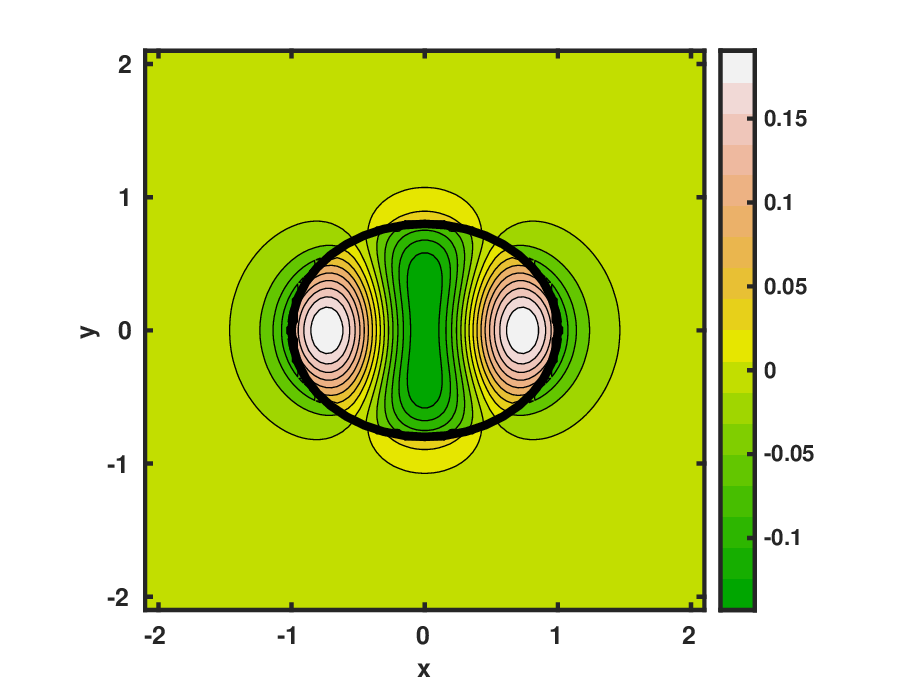}
    \caption{Plots of the real part of the eigenfunctions $w$ within $\mathbb{R}^2\backslash \overline{D}$ and $v$ within $D$ for the ellipse $(1,0.8)$ corresponding to $k_3\approx 3.62554$(left) and $k_4\approx 4.67571$(right).}
    \label{eigenfunctionselli2}
\end{figure}

\subsection{The deformed ellipse}
To continue our numerical investigation, we will consider a more complex scatterer. Next, we use a deformed ellipse which depends on one parameter. The boundary is given in polar coordinates by 
$$\partial D = \big(0.75\cos(t)+\epsilon\cdotp \cos(2t),\sin(t)\big)^\top \,\,\text{ for $t\in [0,2\pi]$ with $\epsilon>0$}$$ 
which has been used before in \cite[Figure 1 \& Equation (4.3)]{cakonikress}. In Table \ref{deformed}, we list the first four transmission eigenvalues for $\epsilon=0.1,0.2,0.3$.
\begin{table}[!ht]
\centering
 \begin{tabular}{r|c|c|c}
 TE & $\epsilon=0.1$ & $\epsilon=0.2$ & $\epsilon=0.3$ \\
\hline
$k_1$ &  1.88515 & 1.89716 & 1.91665 \\
$k_2$ &  3.31667 & 3.34174 & 3.38373 \\
$k_3$ &  3.80574 & 3.77859 & 3.75151 \\
$k_4$ &  4.77845 & 4.86338 & 4.96416 \\
\hline
 \end{tabular}
 \caption{\label{deformed}The first four transmission eigenvalues (TE) for various deformed ellipses with parameter $\epsilon$ using the boundary element collocation method (BEM) with 120 collocation nodes.}
\end{table}
We observe again a monotone increasing behavior of all transmission eigenvalues with respect to the parameter $\epsilon$. However, the area is constant with respect to $\epsilon>0$. Precisely, the area is given $3\pi/4\approx 2.3562$ for all $\epsilon>0$. Hence, we conclude that the monotonicity of the first eigenvalue is not only given through the monotonicity of the area. Note also that we were not able to find any pure complex transmission eigenvalue.
We additionally show in Figure \ref{eigenfunctionsdefo1} and \ref{eigenfunctionsdefo2}  the real part of the first four eigenfunctions corresponding to $k_1\approx 1.91665$, $k_2\approx 3.38373$, $k_3\approx 3.75151$, and $k_4\approx 4.96416$ for a deformed ellipse with $\epsilon=0.3$.
\begin{figure}[H]
    \centering
    \includegraphics[width=0.4\linewidth]{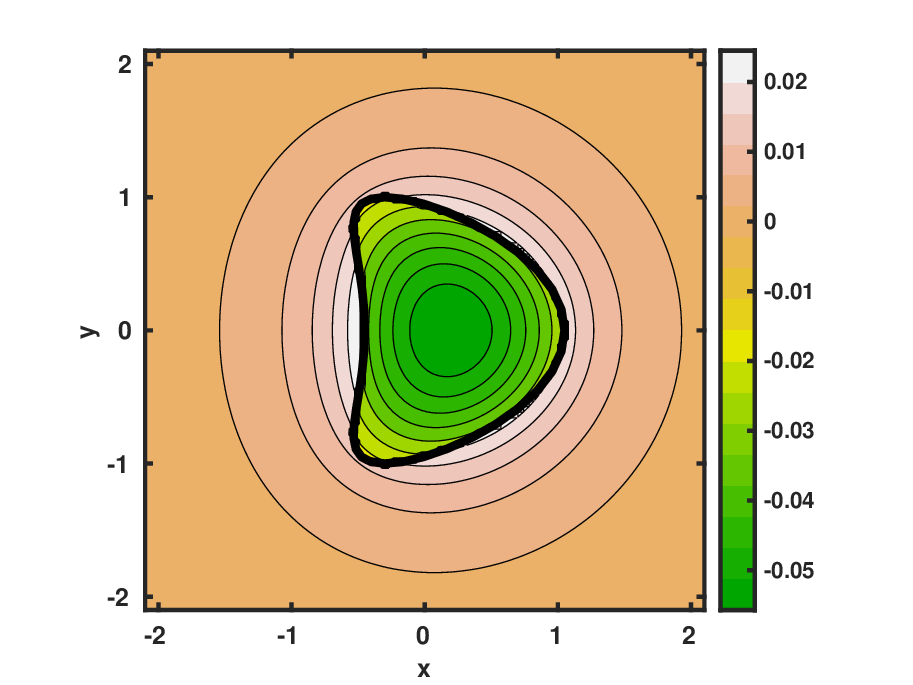}
    \includegraphics[width=0.4\linewidth]{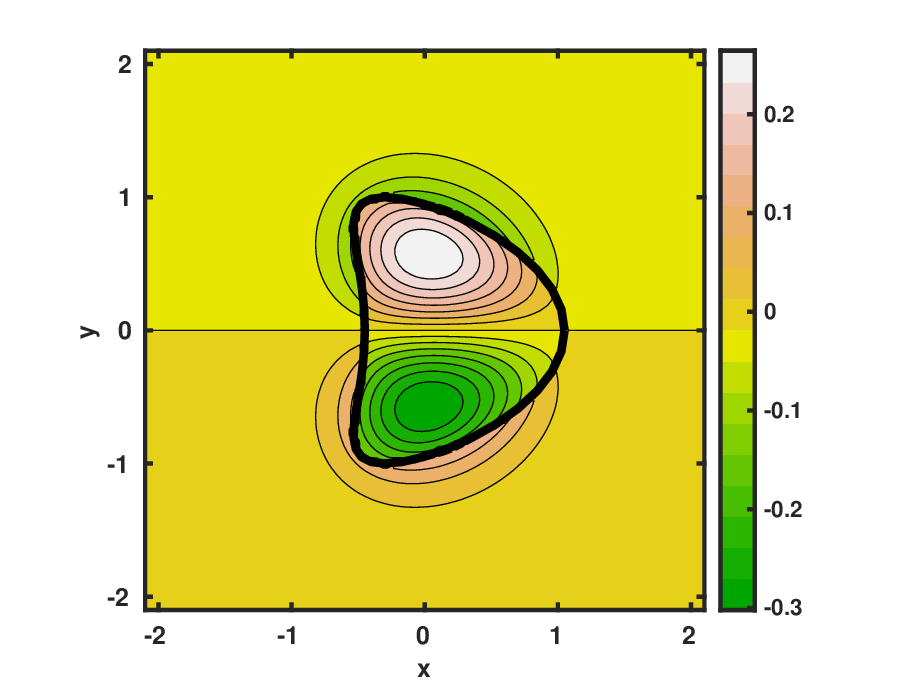}
    \caption{Plots of the real part of the eigenfunctions $w$ within $\mathbb{R}^2\backslash \overline{D}$ and $v$ within $D$ for the deformed ellipse with $\epsilon=0.3$ corresponding to $k_1\approx 1.91665$(left) and $k_2\approx 3.38373$(right).}
    \label{eigenfunctionsdefo1}
\end{figure}

\begin{figure}[H]
    \centering
    \includegraphics[width=0.4\linewidth]{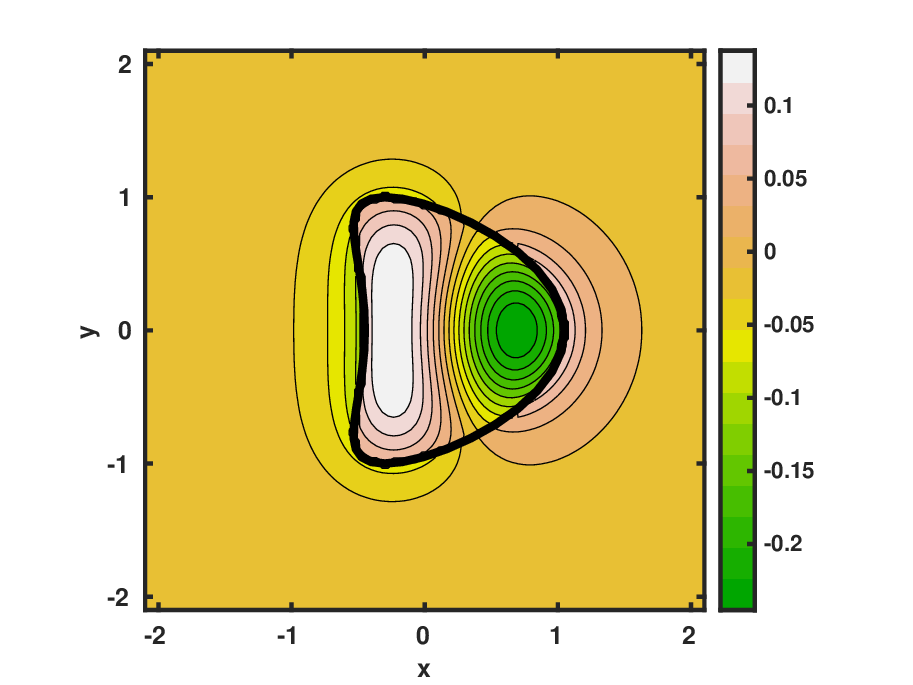}
    \includegraphics[width=0.4\linewidth]{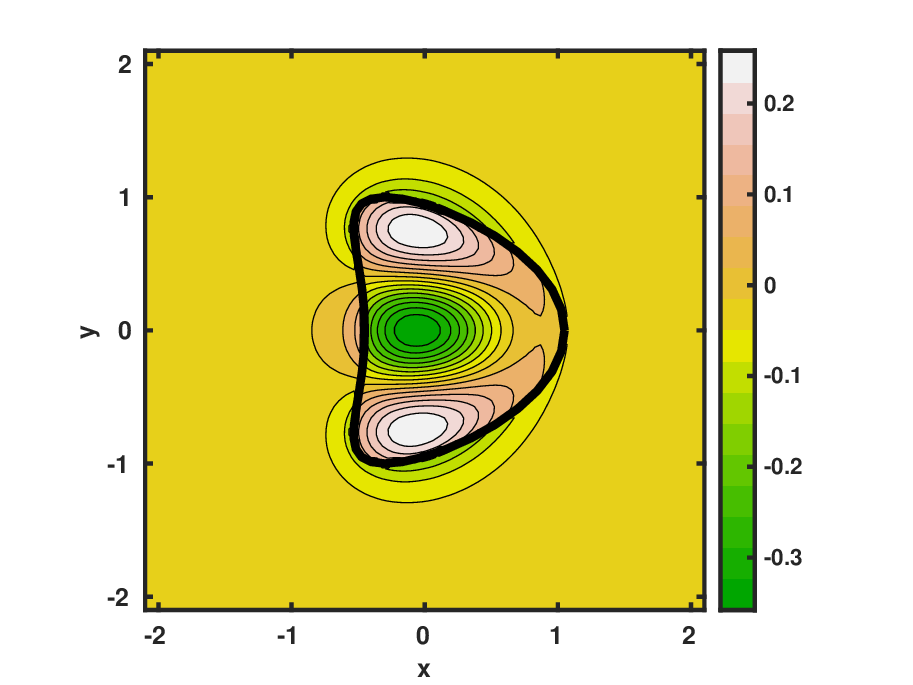}
    \caption{Plots of the real part of the eigenfunctions $w$ within $\mathbb{R}^2\backslash \overline{D}$ and $v$ within $D$ for the deformed ellipse with $\epsilon=0.3$ corresponding to $k_3\approx 3.75151$(left) and $k_4\approx 4.96416$(right).}
    \label{eigenfunctionsdefo2}
\end{figure}

\subsection{Determination from the Data}
Now, we wish to give numerical validation of Theorem \ref{determineEig}. For other problems this has been shown to be an effective method for computing the transmission eigenvalues without much a priori knowledge of the scatterer. Indeed, one only needs to have an idea of the location of the scatterer to apply Theorem \ref{determineEig}. This will be done by computing the far field data for different wave numbers and solving the far field equation 
$$ \int_{\mathbb{S}^1} u^\infty (\hat{x}, d;k) g(d) \,  \dd s(d) =  \text{e}^{-{\mathrm{i}k} \hat{x} \cdot z} \quad \text{for $\hat{x} \in\mathbb{S}^1$ and  $z \in D$.}$$
In order to compute the far field data $u^\infty (\hat{x}, d ; k)$, we use the numerical method studied in \cite{DongLi24}. This method also uses a system of boundary integral equations for computing the far field pattern. The derivation and discretization is similar to the previous section so to avoid repetition see Section 5 of \cite{DongLi24} for details (see also \cite{DSM-BH24}).

For our examples, we define the multi-static far field matrix
$$\textbf{F}=\Big[u^\infty(\hat{x}_i , d_j ; k)\Big]^{64}_{i,j=1}.$$ 
Note that to discretize the boundary of the unit disk $\mathbb{S}^1$, we define 
$$\hat{x}_i=d_i= \big(\cos(\theta_i),\sin(\theta_i)\big)^\top,\quad \theta_i=2\pi(i-1)/64,\quad i=1,\dots,64.$$ 
Thus, we have that the matrix $\textbf{F} \in \C^{64 \times 64}$ corresponds to the far field data for $64$ incident and observation directions. In reality one does not have access to accurate measurements, so to simulate measurement error in the data, noise is added to the far field matrix such that 
$$\mathbf{F}^{\delta}=\Big[ \mathbf{F}(i,j) \big(1+\delta\mathbf{E}(i,j) \big)\Big]^{64}_{i,j=1}.$$
Here, the corresponding error matrix $\mathbf{E} \in \C^{64 \times 64}$ consists of (uniformly distributed) random values whose real and imaginary part are taken in the interval $[-1, 1]$. The matrix $\mathbf{E}$ is then normalized in the Frobenius norm with $0<\delta < 1$ representing the relative noise level added to the synthetic data. Then, the far field equation becomes 
$$\mathbf{F}^{\delta} {\bf g}_z =  \boldsymbol{\phi}_z$$ 
where $\boldsymbol{\phi}_z=\big(\text{e}^{-\text{i}k z \cdot \hat{x}_1},\dots, \text{e}^{-\text{i}k z \cdot \hat{x}_{64}}\big)^\top$ for $z\in D$. For 150 equally spaced wave numbers $k \in [1 , 5 ]$ we solve the discretized far field equation at 20 different $x \in D$ and average the norms to get the following results. The now discretized far field equation is solved via Tikhonov regularization with Morozov discrepancy principle.

First, we recover the transmission eigenvalues for the unit disk. Recall that the values of these eigenvalues are given in Table \ref{TEcircleTable}. In Figure \ref{recon-circle}, we provide a numerical example that shows that the transmission eigenvalues can be accurately recovered from the far field data even when one has noisy data. The data tip in Figure \ref{recon-circle} closely match the previously computed eigenvalues. 
 \begin{figure}[h]
    \centering
    \includegraphics[width=0.475\linewidth]{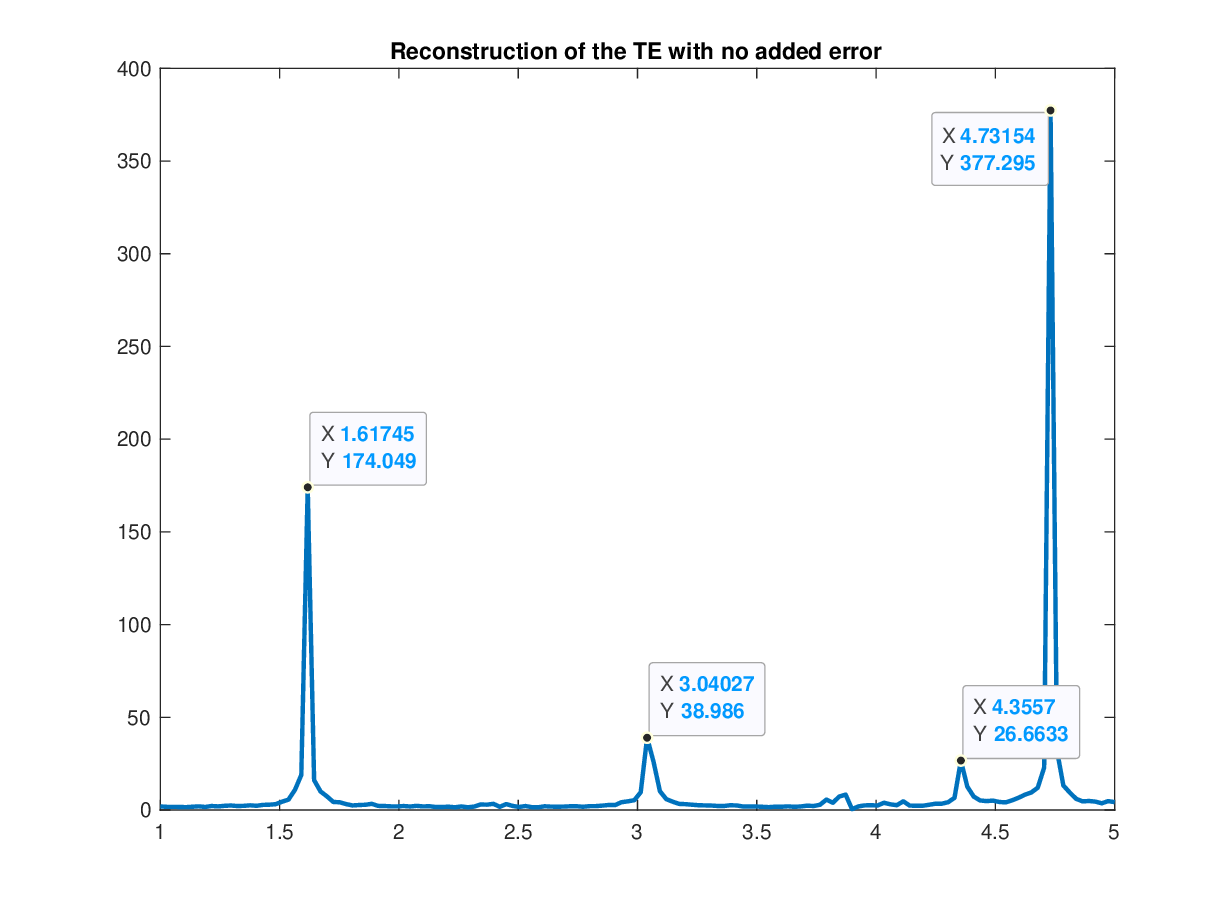}
    \includegraphics[width=0.475\linewidth]{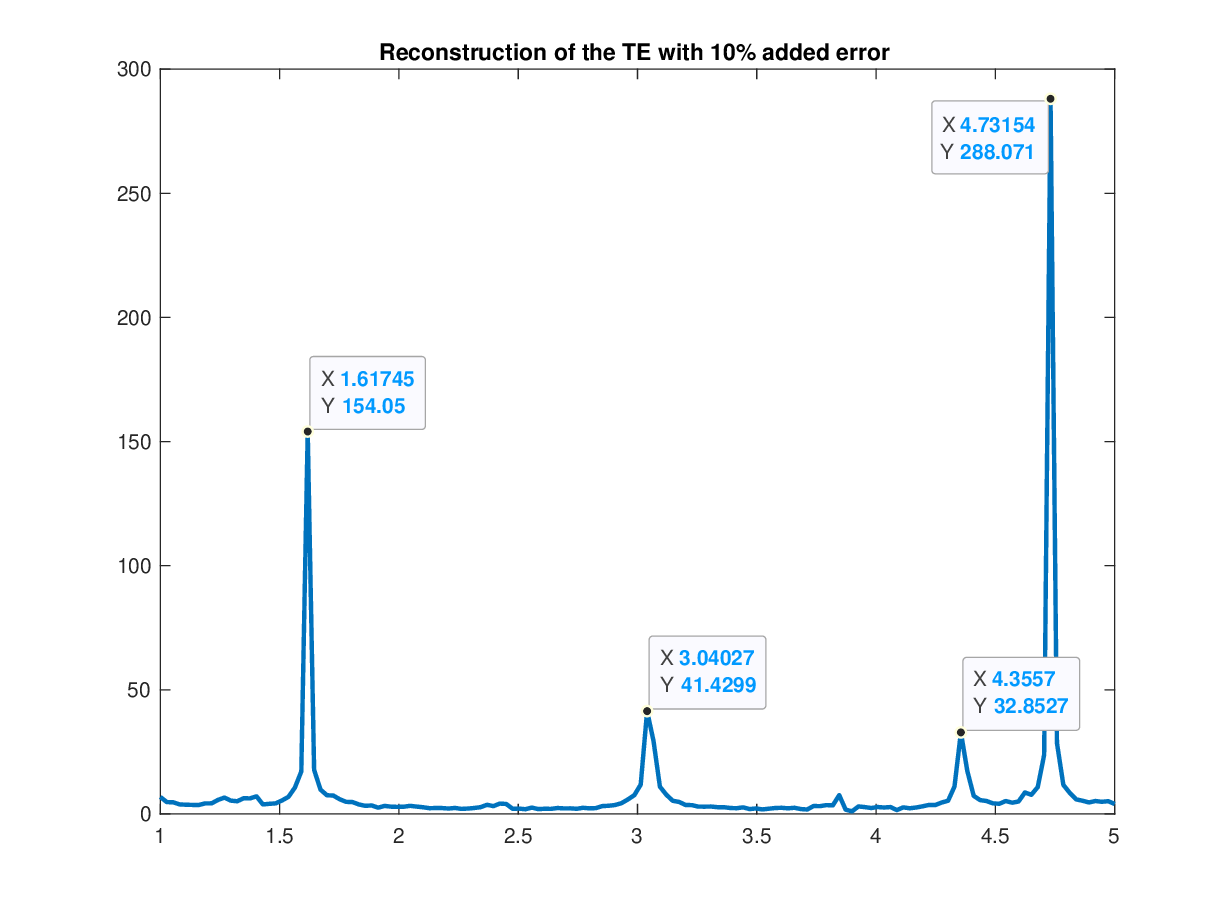}
    \caption{The recovered transmission eigenvalues for the unit disk from the far field data. Recall that the first three distinct eigenvalues are given by $k\approx1.61464, \, 3.05164, \, 4.36453$ from Table \ref{TEcircleTable}. }
    \label{recon-circle}
\end{figure}

With this, we see that the new  transmission eigenvalues studied here can also be numerically computed via the far field data. Now, we give a numerical example for recovering the transmission eigenvalues for another scatterer. Here, we will consider a different scatterer than the previous sections. The scatterer that we will consider is the peanut--shaped region which is given by
$$\partial D=0.5\sqrt{3\cos(t)^2+1}\big(\cos(t),\sin(t)\big)^\top  \,\, \text{ for $t\in [0,2\pi]$.}$$
Notice that the area for the peanut--shaped region is approximately $1.9635$ which is clearly less than that of the unit disk. The first four eigenvalues are approximately given in Table \ref{peanut} computed via the BEM from the previous section.
\begin{table}[!ht]
\centering
 \begin{tabular}{r|c}
 TE &  \\
\hline
$k_1$ &  2.13093  \\
$k_2$ &  3.41900  \\
$k_3$ &  4.70289  \\
$k_4$ &  4.89266  \\
\hline
 \end{tabular}
 \caption{\label{peanut}The first four transmission eigenvalues (TE) for the peanut--shaped region using the boundary element collocation method (BEM) with 120 collocation nodes.}
\end{table}
From Figure \ref{recon-Peanut} (also see Table \ref{peanut}), we have that the first eigenvalue for the peanut--shaped region is larger than for unit disk which seems to give more evidence that the transmission eigenvalues may be monotone with respect to the area. 

 \begin{figure}[H]
    \centering
    \includegraphics[width=0.42\linewidth]{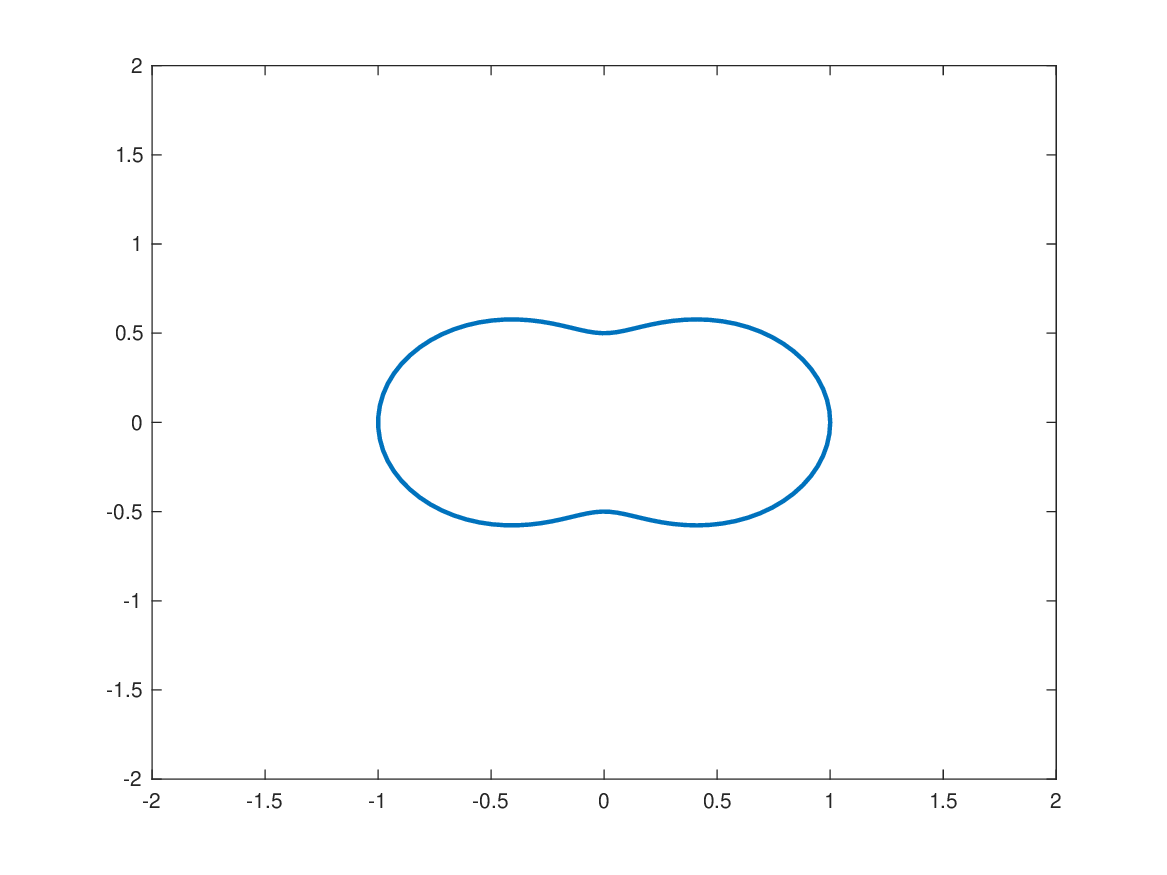}
    \includegraphics[width=0.42\linewidth]{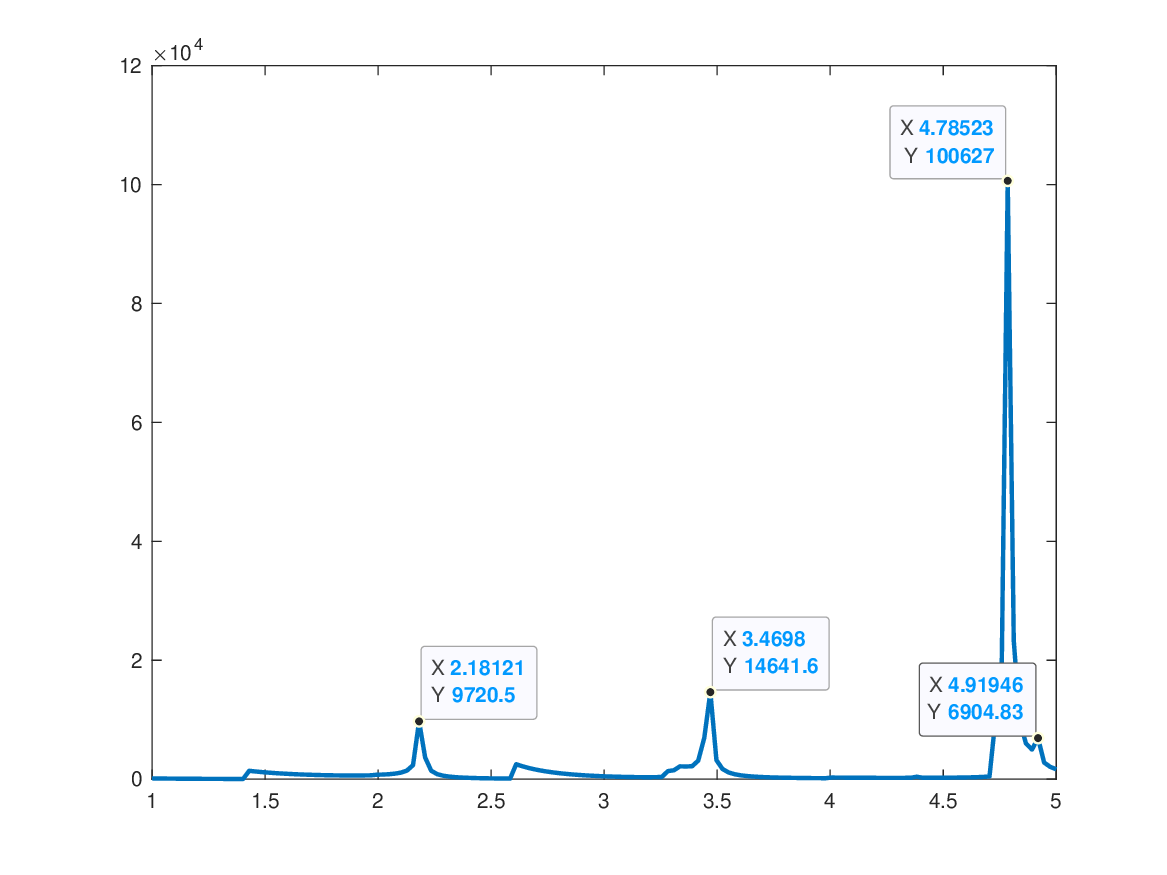}
    \caption{The recovered transmission eigenvalues for the peanut--shaped region from the far field data. Left: peanut--shaped scattering region; Right: reconstruction of the transmission eigenvalues. The reconstructed transmission eigenvalues are given by $k\approx$ 2.18121, 3.4698, 4.78523, and $4.91946$ which are approximately the values given  in Table \ref{peanut}.}
    \label{recon-Peanut}
\end{figure}

\section{Conclusion and Outlook}
In conclusion, we have derived and initiated the study of the clamped transmission eigenvalue problem \eqref{tep1}--\eqref{tep2}. This is a new eigenvalue problem derived from scattering in a thin elastic plate. Some of the analysis used here is similar as for other problems but there are many differences. The main one is that this is an eigenvalue problem for partial differential equation in all of free space. This is similar to the exterior transmission eigenvalues defined in $\exD$ that where initially derived in \cite{ccmlsm} (see also \cite{z-extTE}). We have proven discreteness of the eigenvalues as well as the fact that they can be recovered from the far field data. The recovery of the eigenvalues is interesting due to the fact that the far field data contains no information for the $u_M$ part of the solution to the direct problem. 

We have also provided numerical experiments for the clamped transmission eigenvalues using boundary integral equations as well as from far field data for different wave numbers. Numerically we see that the reconstruction of the eigenvalues seems to be stable with respect to noisy data, if one uses a suitable regularization technique. Also, from our numerical experiments we seem to have that the eigenvalues are monotone with respect to the measure of the region. Indeed, from our calculations for the unit disk with area 3.14159 we have that $k_1\approx 1.61464$, peanut with area 1.9635 we recovered $k_1\approx 2.13093$, and ellipse $(1,0.5)$ with area 1.57079 we obtained $k_1\approx 2.40418$. Therefore, studying the monotonicity further along with other isoperimetric inequalities will be part of our future investigations.



\begin{thebibliography}{99}


\bibitem{adams-embedding}
\newblock R. Adams, 
\newblock Compact Sobolev imbeddings for unbounded domains, 
\newblock {\it Pacific Journal of Mathematics}, {\bf 32(1)} 1--7 (1970).


\bibitem{ap-invTE}
\newblock T.  Aktosun and V. Papanicolaou, 
\newblock Reconstruction of the wave speed from transmission eigenvalues for the spherically symmetric variable-speed wave equation, 
\newblock {\it Inverse Problems}, {\bf 29} 065007 (2013).


\bibitem{beyn}
\newblock W.-J. Beyn. 
\newblock An integral method for solving nonlinear eigenvalue problems, 
\newblock {\it Linear Algebra and its Applications}, {\bf 436(10)}, 3839--3863 (2012).


\bibitem{near-lsmBH}
\newblock L. Bourgeois,  and  A. Recoquillay, 
\newblock The Linear Sampling Method for Kirchhoff-Love infinite plates, 
\newblock {\it Inverse Problems and Imaging},  {\bf 14(2)}, 363--384 (2020). 
 

\bibitem{bck-invTE}
\newblock S.-A. Buterin, A.-E. Choque-Rivero and M.-A. Kuznetsova,
\newblock On a regularization approach to the inverse transmission eigenvalue problem, 
\newblock {\it Inverse Problems}, {\bf 36} 105002 (2020). 


 
 
\bibitem{Cakoni-Colton-book} 
\newblock F. Cakoni and D. Colton,
\newblock {\it A qualitative approach to inverse scattering theory},
\newblock {Springer}, {New York},  (2013).


\bibitem{cchlsm}
\newblock F. Cakoni, D. Colton, and H. Haddar,
\newblock On the determination of Dirichlet or transmission eigenvalues from far field data,
\newblock {\it C. R. Acad. Sci. Paris}, {\bf 348}, 379--383 (2010).


\bibitem{ccmlsm}
\newblock F. Cakoni, D. Colton and S. Meng,
\newblock The inverse scattering problem for a penetrable cavity with internal measurements, 
\newblock {\it AMS Contemporary Mathematics}, {\bf 615}, 71--88 (2014).


\bibitem{cgh-TE}
\newblock F. Cakoni, D. Gintides and H. Haddar, 
\newblock The existence of an infinite discrete set of transmission eigenvalues,
\newblock {\it SIAM J. Math. Analysis}, {\bf 42(1)}, 237--255 (2010).


\bibitem{Cakoni-TE}
\newblock F. Cakoni and H. Haddar,
\newblock {On the existence of transmission eigenvalues in an inhomogeneous medium},
\newblock{\it Applicable Analysis}, {\bf 88(4)}, 475--493 (2009).


\bibitem{homogeniz-TE}
\newblock F. Cakoni, H. Haddar, and I. Harris, 
\newblock Homogenization of the transmission eigenvalue problem for periodic media and application to the inverse problem,
\newblock {\it Inverse Problems and Imaging},  {\bf 9(4)}, 1025--1049 (2015). 


\bibitem{cakonikress}
\newblock F. Cakoni and R. Kress, 
\newblock A boundary integral equation method for the transmission eigenvalue problem,
\newblock {\it Applicable Analysis}, {\bf 96(1)}, 23--38 (2017).


\bibitem{te-2cbc}
\newblock  R. Ceja Ayala, I. Harris, A. Kleefeld, and N.  Pallikarakis,
\newblock Analysis of the transmission eigenvalue problem with two conductivity parameters, 
\newblock {\it Applicable Analysis}, {\bf 103(1)}, 211--239 (2024). 


\bibitem{iterative-BH23}
\newblock  Y. Chang and  Y. Guo, 
\newblock An optimization method for the inverse scattering problem of the biharmonic wave,
\newblock {\it Communications on Analysis and Computation}  {\bf 1(2)}, 168--182 (2023). 


\bibitem{ColtYj}
\newblock D. Colton and Y.-J. Leung, 
\newblock The existence of complex transmission eigenvalues for spherically stratified media,
\newblock {\it Applicable Analysis}, {\bf 96(1)}, 39--47 (2017). 


\bibitem{ColtYjShixu}
\newblock D. Colton, Y.-J. Leung, and S. Meng,
\newblock Distribution of complex transmission eigenvalues for spherically stratified media,
\newblock {\it Applicable Analysis}, {\bf 31}, 035006 (2015). 


\bibitem{coltonkress}
\newblock D. Colton and R. Kress, 
\newblock {\it Inverse acoustic and electromagnetic scattering theory}. Springer, 2 edition, 1998.


\bibitem{applicationref2} 
\newblock Y. Deng, J. Li and H. Liu, 
\newblock On identifying magnetized anomalies using geomagnetic monitoring.
\newblock {\it Arch. Ration. Mech. Anal.}, {\bf 231(1)} 153--187 (2019).


\bibitem{applicationref3} 
\newblock Y. Deng, J. Li and H. Liu, 
\newblock On identifying magnetized anomalies using geomagnetic monitoring within a magnetohydrodynamic model. 
\newblock {\it Arch. Ration. Mech. Anal.}, {\bf 235(1)}  691--721 (2020).


\bibitem{applicationref4} 
\newblock Y. Deng, J. Li and G. Uhlmann. 
\newblock On an inverse boundary problem arising in brain imaging, 
\newblock {\it J. Differential Equations}, {\bf 267(4)} 2471--2502  (2019).


\bibitem{DongLi24}
\newblock H. Dong and P. Li,
\newblock {A Novel boundary integral formulation for the biharmonic wave scattering problem,} 
\newblock{\it Journal of Scientific Computing}, {\bf 98}, 42 (2024).


\bibitem{DongLi-Unique}
\newblock H. Dong and P. Li,
\newblock {Uniqueness of an inverse cavity scattering problem for the time-harmonic biharmonic wave equation,} 
\newblock{\it Inverse Problems}, {\bf 40}, 065011 (2024).


\bibitem{gp-invTE}
\newblock D.  Gintides and N. Pallikarakis, 
\newblock The inverse transmission eigenvalue problem for a discontinuous refractive index,
\newblock {\it Inverse Problems}, {\bf 33} 055006 (2017). 


\bibitem{LSM-BHclamped}
\newblock J. Guo, Y. Long, Q. Wu, and J. Li,
\newblock On direct and inverse obstacle scattering problems for biharmonic waves,
\newblock {\it Inverse Problems}, {\bf 40}, 125032 (2024). 


\bibitem{mypaper1} 
\newblock I. Harris, F. Cakoni, and J. Sun, 
\newblock{Transmission eigenvalues and non-destructive testing of anisotropic magnetic materials with voids,} 
\newblock {\it Inverse Problems}, {\bf 30}, 035016 (2014).


\bibitem{te-cbc2}
\newblock I. Harris and A. Kleefeld, 
\newblock The inverse scattering problem for a conductive boundary condition and transmission eigenvalues, 
\newblock {\it Applicable Analysis}, {\bf 99(3)}, 508--529 (2020).


\bibitem{DSM-BH24}
\newblock I. Harris, H. Lee and P. Li,
\newblock {Direct sampling for recovering a clamped cavity from biharmonic far field data,} 
\newblock arXiv:2409.04870 (2025).


\bibitem{aniso-periodic}
\newblock I. Harris, D.-L. Nguyen, J. Sands, and T. Truong 
\newblock On the inverse scattering from anisotropic periodic layers and transmission eigenvalues,
\newblock{ \it Applicable Analysis}, {\bf 101(8)}, 3065--3081 (2022).


\bibitem{inout-modified}
\newblock H. Haddar, M. Khenissi,  and M. Mansouri, 
\newblock Inside-Outside duality for a class of modified transmission eigenvalues,
\newblock {\it Inverse Problems and Imaging}, {\bf 17(4)}, 798--816 (2023).   


\bibitem{armin}
\newblock A. Kirsch and A. Lechleiter,
\newblock The inside-outside duality for scattering problems by inhomogeneous media,
\newblock {\it Inverse Problems}, {\bf 29}, 104011 (2013).

\bibitem{Kleefeld}
\newblock A. Kleefeld,
\newblock A numerical method to compute interior transmission eigenvalues,
\newblock{\it Inverse Problem}, {\bf29(10)}, 104012 (2013).

\bibitem{armin2}
\newblock A. Lechleiter and M. Rennoch, 
\newblock Inside--outside duality and the determination of electromagnetic interior transmission eigenvalues,
\newblock {\it SIAM J. Math. Analysis}, {\bf 47(1)}, 684--705 (2015).


\bibitem{poh-invsource}
\newblock J. Li, P. Li, X. Wang, and G. Yang, 
\newblock Inverse random potential scattering for the polyharmonic wave equation using far-field patterns, 
\newblock arXiv:2407.15681 (2024).


\bibitem{inout-iso}
\newblock S. Peters and A. Kleefeld,
\newblock Numerical computations of interior transmission eigenvalues for scattering objects with cavities,
\newblock {\it Inverse Problems}, {\bf 32}, 045001 (2016).


\bibitem{ComplexTrajectories}
\newblock L. Pieronek and and A. Kleefeld, 
\newblock On trajectories of complex-valued interior transmission eigenvalues,
\newblock {\it Inverse Problems and Imaging}, {\bf 18(2)}, 480--516 (2024).   


\bibitem{sun-reconTE}
\newblock J. Sun, 
\newblock Estimation of the transmission eigenvalue and the index of refraction using Cauchy data, 
\newblock {\it Inverse Problems}, {\bf 27}, 015009 (2011).


\bibitem{z-extTE}
\newblock Y. Zhang,
\newblock Spectral properties of exterior transmission problem for spherically stratified anisotropic media,
\newblock{\it Applicable Analysis}, {\bf 100(8)},1668--1692  (2021).


\end{thebibliography}
\end{document}